\documentclass[twoside]{irmaems}
\usepackage{amssymb} 
\usepackage{amsmath} 
\usepackage{latexsym}
\usepackage{graphicx}

\renewcommand\theequation{\arabic{section}.\arabic{equation}}

\theoremstyle{definition} 

 \newtheorem{definition}{Definition}[section]
 \newtheorem{remark}[definition]{Remark}

\theoremstyle{plain}      

 \newtheorem{proposition}[definition]{Proposition}
 \newtheorem{theorem}[definition]{Theorem}
 \newtheorem{corollary}[definition]{Corollary}
 \newtheorem{lemma}[definition]{Lemma}

\newtheorem*{conjecture}{Conjecture}
\newtheorem{oqu}{Open question}

\usepackage{indentfirst}

\usepackage{color}
\definecolor{darkred}{rgb}{0.9,0.,.2}
\definecolor{darkblue}{rgb}{0.,0.,.6}
\definecolor{darkgreen}{rgb}{0.,.6,0.1}
\usepackage[colorlinks=true,urlcolor=darkblue,linkcolor=darkred,citecolor=darkgreen]{hyperref}

\usepackage{makeidx}

\newcommand{\N}{\mathbb{N}}
\newcommand{\Z}{\mathbb{Z}}

\newcommand{\R}{\mathbb{R}}
\newcommand{\C}{\mathcal{C}} 

\newcommand{\A}{\mathbb{A}}

\renewcommand{\ss}{\mathrm{SL_{d+1}(\mathbb{R})}}

\newcommand{\SO}{\mathrm{SO_{d,1}(\mathbb{R})}}

\newcommand*{\so}[1]{\mathrm{SO}_{#1,1}(\mathbb{R})}
\newcommand*{\s}[1]{\mathrm{SL}_{#1}(\mathbb{R})}

\newcommand{\PG}{\mathrm{PGL_{d+1}(\mathbb{R})}}

\newcommand{\LG}{\Lambda_{\Gamma}}
\newcommand{\G}{\Gamma}
\newcommand{\g}{\gamma}

\renewcommand{\C}{\mathcal{C}}
\newcommand{\U}{\mathcal{U}}

\newcommand{\E}{\mathcal{E}}
\newcommand{\GG}{\mathcal{G}}
\renewcommand{\AA}{\mathcal{A}}

\renewcommand{\O}{\Omega}

\newcommand{\dO}{\partial \Omega}
\renewcommand{\d}{d_{\Omega}}

\renewcommand{\S}{\mathbb{S}}

\renewcommand{\A}{\mathbb{A}}

\newcommand{\PP}{\mathbb{P}}

\newcommand{\Quo}{\Omega/\!\raisebox{-.90ex}{\ensuremath{\,\Gamma}}}

\newcommand*{\Quotient}[2]{\ensuremath{#1/\!\raisebox{-.90ex}{\,\ensuremath{#2}}}}

\newcommand{\Aut}{\mathrm{Aut}}

\newcommand{\Hom}{\mathrm{Hom}}

\newcommand{\Stab}{\mathrm{Stab}}

\newcommand{\Isom}{\mathrm{Isom}}

\makeindex

\begin{document}

\title{Around groups in Hilbert Geometry}

\author{Ludovic Marquis 
}

\address{
Institut de Recherche Math\'ematique de Rennes\\
email:\,\tt{ludovic.marquis@univ-rennes1.fr}
}

\maketitle





\par{
In this chapter, we survey groups of projective transformation acting on a Hilbert geometry. 
}
\\
\par{
Hilbert geometries were introduced by Hilbert as examples of geodesic metric spaces where straight lines are geodesics. We will completely forget this story. We will take Hilbert geometry as a very simple recipe for metric spaces with several different flavours.
}
\\
\par{
Hilbert geometries are Finsler manifold, so it is not easy to say that they are non-positively curved, but we hope that at the end of this text the reader will have noticed some flavours of non-positively curved manifolds and will start to see them as ``damaged non-positively curved manifold''.
}
\\
\par{
The most interesting examples of Hilbert geometries in the context of geometric group theory are called, following Vey in \cite{MR0283720}, \emph{divisible convex sets}\index{divisible convex}\index{convex!divisible}. These are those properly convex open subsets $\O$ of the real projective space $\PP^d=\PP^d(\R)$ such that there exists a discrete subgroup $\G$ of the group $\PG$ of projective transformation which preserves $\O$ and such that the quotient $\Quo$ is compact. In 2006, Benoist wrote a survey \cite{MR2464391} of divisible convex sets and in 2010, Quint wrote a survey \cite{MR2648674} of the work of Benoist on divisible convex sets.
}
\\
\par{
Thus, we will not concentrate on divisible convex sets since the survey of Benoist does this job very well, even if we consider divisible \emph{and} quasi-divisible convex sets\index{quasi-divisible convex}\index{convex!quasi-divisible}\footnote{Those for which the quotient $\Quo$ is of finite volume rather than compact.} as the most important class of convex sets. We want to describe the groups that appear in Hilbert geometry without restriction and also how groups can be used for the purpose of Hilbert geometry.
}
\\
\par{
The first two parts of this survey are very elementary. Their goal is to make the reader familiar with the possible automorphisms of a convex set from a matrix point of view and a dynamical point of view.
}
\\
\par{
The third part presents existence results on convex sets with a ``large'' group of symmetries. This part presents the examples that motivated this study. The reader in quest of motivation should skip Part 2 to go straight to the third part.
}
\\
\par{
The remaining seven parts are roughly independent. They support the claim that geometric group theory mixed with Hilbert geometry meet (and need) at least: \emph{differential geometry, convex affine geometry, real algebraic group theory, metric geometry, moduli spaces, hyperbolic geometry, symmetric spaces, Hadamard manifolds and geometry on manifolds}. We have tried to give sketchy proofs for these parts. We do not report on the links with \emph{Coxeter group theory} and \emph{partial differential equations}.
}

\tableofcontents

\subsection*{Context}\label{ref_volume}

\par{
We will work with the real projective space $\PP^d(\R)=\PP^d$ of dimension $d$ (i.e the space of lines of the real vector space $\R^{d+1}$). An \emph{affine chart}\index{affine chart} of $\PP^d$ is the complement of a projective hyperplane. Every affine chart carries a natural structure of an affine space. A \emph{convex}\index{convex} subset of the projective space $\PP^d$ is a subset of $\PP^d$ which is either $\PP^d$ or is included in an affine chart $\A$ and is convex in this affine chart in the usual sense.
}
\\
\par{
A convex subset $C$ of $\PP^d$ is \emph{properly convex}\index{properly convex}\index{convex!properly convex} when there exists an affine chart such that $C$ is bounded in it or, equivalently, when $C$ does not contain any affine line. Our playground will be a \emph{properly convex open set}, also called a \emph{convex body}\index{convex!body}. We will always denote a properly convex open subset of $\PP^d$ by the letter $\O$ with a subscript if necessary.
}
\\
\par{
On a properly convex open set, one can define the Hilbert distance \index{Hilbert distance}\index{distance!Hilbert}. Given $x \neq y \in \O$, let $p,q$ be the intersection points of $(xy)$ with the boundary $\partial \O$ of $\O$ in such a way that $x$ is between $p$ and $y$ and $y$ between $x$ and $q$ (see Figure \ref{disttt}). We set:$$
\begin{array}{ccc}
d_{\O}(x,y) = \displaystyle\frac{1}{2}\ln \big([p:x:y:q]\big) = \displaystyle\frac{1}{2}\ln \bigg(\frac{|py|\cdot |
qx|}{|px| \cdot |qy|} \bigg) & \textrm{and} &
d_{\O}(x,x)=0,
\end{array}
$$
\begin{center}
\begin{figure}[h!]
\centering
\includegraphics[width=7cm]{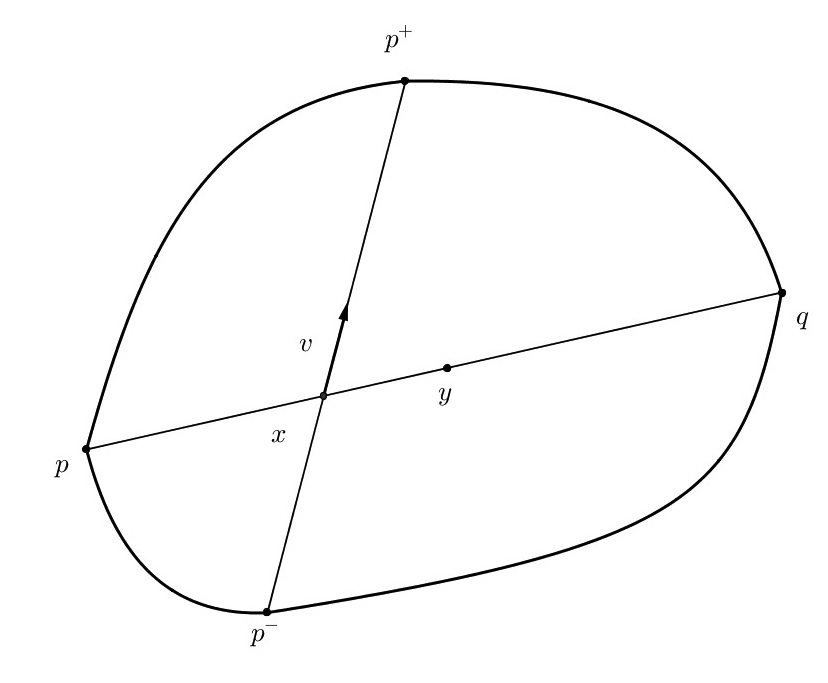}
\caption{Hilbert distance\label{disttt}}
\end{figure}
\end{center}
where the quantity $[p:x:y:q]$ is the cross-ratio of the four points $p,x,y,q$,
and $| \cdot |$ is any Euclidean norm on any affine chart $\A$ containing the closure\footnote{In fact, one should remark that if we allow the symbol $\infty$ in our computation, then we can define $d_{\O}$ on any affine chart containing $\O$.} $\overline{\O}$ of $\O$.

The cross-ratio is a projective notion. Thus it is clear that $d_{\O}$ does not depend on $\A$ or on the choice of the Euclidean norm on $\A$. We also note that the distance $d_{\O}$ is invariant by the group of projective automorphisms preserving $\O$.
}
\\
\par{
The metric space $(\O,d_{\O})$ is geodesic\footnote{A metric space is \emph{geodesic}\index{geodesic metric space}\index{metric space!geodesic} when one can find a geodesic joining any two points.}, proper\footnote{A metric space is \emph{proper}\index{proper metric space}\index{metric space!proper} when the closed balls are compact, and so a proper space is complete.} and the topology induced by the Hilbert distance and the projective space coincide. But this space is not uniquely geodesic\footnote{A geodesic metric space is \emph{uniquely geodesic}\index{uniquely geodesic metric space}\index{metric space!uniquely geodesic} when the geodesic between any two points is unique.} in general (see Proposition \ref{seg_geo}).
}
\\
\par{
This distance $d_{\O}$  is called the Hilbert distance and has the good taste of coming from a Finsler metric on $\O$ defined by a very simple formula. Let $x$ be a point in $\O$ and $v$ a vector in the tangent space $T_x \O$ of $\O$ at $x$; the quantity  $\left. \frac{d}{dt}\right| _{t=0} d_{\O}(x,x+tv)$ is homogeneous of degree one in $v$, therefore it defines a Finsler metric $F(x,v)$ on $\O$. Moreover, if we choose an affine chart $\A$ containing $\O$ and any Euclidean norm $|\cdot|$ on $\A$, we get (see Figure \ref{disttt}):
}

$$F_{\O}(x,v) = \left. \frac{d}{dt}\right| _{t=0} d_{\O}(x,x+tv) = \frac{|v|}{2}\Bigg(\frac{1}{|xp^-|} + \frac{1}{| xp^+|} \Bigg)$$

\par{
The Finsler structure gives rise to a measure $\mu_{\O}$ on $\O$ which is absolutely continuous with respect to the Lebesgue measure, called the \emph{Busemann volume}\index{Busemann volume}\index{volume!Busemann}. To define it, choose an affine chart $\A$ containing $\O$, a Euclidean norm $|\cdot|$ on $\A$ and let $\mathrm{Leb}$ be the Lebesgue measure on $\A$ normalised by the fact that the volume of the unit cube is 1. The Busemann volume\footnote{There is a second volume which is often used in Finsler geometry, the so-called \emph{Holmes-Thompson volume} $\mu_{HT}$, defined by the following formula:
$$
\mu_{HT}(\mathcal{A})=\int_{\mathcal{A}} \frac{\mathrm{Leb}(B^{\mathrm{CoTang}}_x(1))}{\omega_d}d \mathrm{Leb}(x)
$$
\noindent where $B^{\mathrm{CoTang}}_x(1))$ is the unit ball of the cotangent space of $\O$ at $x$ for the dual norm induced by $(T_x\O,F_{\O}(x,\cdot))$.
} of a Borel set $\mathcal{A} \subset \O$ is then defined by the following formula:
}
$$
\mu_{\O}(\mathcal{A})=\int_{\mathcal{A}} \frac{\omega_d}{\mathrm{Leb}(B^{\mathrm{Tang}}_x(1))}d \mathrm{Leb}(x)
$$
\par{
\noindent where $\omega_d$ is the Lebesgue volume of the unit ball of $\A$ for the metric $|\cdot|$ and $B^{\mathrm{Tang}}_x(1)$ is the unit ball of the tangent space $T_x \O$ of $\O$ at $x$ for the metric $F_{\O}(x,\cdot)$.
}
\\
\par{
A charming fact about Hilbert geometries is that they are comparable between themselves. Indeed, if $\O_1 \subset \O_2$ then we can compare the distance, the balls, etc... of those two properly convex open sets. The moral statement is the following:
}

\emph{
The Hilbert distance is decreasing with $\O$. The ball (in $\O$ and in the tangent space) are increasing with $\O$. The Busemann volume is decreasing with $\O$.
}

The precise statement is the following:

\begin{proposition}\label{compa}
Let $\O_1,\O_2$ be two properly open convex subset such that $\O_1 \subset \O_2$, then:
\begin{itemize}
\item For every $x,y \in \O_1$, $d_{\O_1}(x,y) \leqslant d_{\O_2}(x,y)$.
\item For every $x \in \O_1$, for every $v \in T_x \O_1=T_x \O_2$, $F_{\O_1}(x,v) \leqslant F_{\O_2}(x,v) $. 
\item For any Borel set $\AA$ of $\O_1$, $\mu_{\O_2}(\AA) \leqslant \mu_{\O_1}(\AA)$.
\end{itemize}
\end{proposition}

\par{
Precisely, for the Busemann volume: if $\O_1 \subset \O_2$, then for any Borel set $\AA$ of $\O_1$, we have $\mu_{\O_2}(\AA) \leqslant \mu_{\O_1}(\AA)$.
}
\\
\par{
We will see that regularity conditions of $\dO$ have a crucial impact on the geometry of $(\O,d_{\O})$. A properly convex open set is \emph{strictly convex}\index{strictly convex}\index{convex!strictly} if there does not exist any non-trivial segment in his boundary $\dO$. A properly convex open set has \emph{$\C^1$ boundary}\index{convex! with $\C^1$ boundary} if the hypersurface $\dO$ is a submanifold of class $\C^1$; since $\O$ is convex this means that at every point $p\in \dO$, there is a unique supporting hyperplane\footnote{A hyperplane $H$ is a \emph{supporting hyperplane}\index{supporting hyperplane}\index{hyperplane!supporting} at $p\in \dO$ when $H \cap \O= \varnothing$ and $p\in H$.} at $p$ for $\O$. We will see that these two properties are dual to each other (see Part \ref{duality1}). A properly convex set which verifies both properties is called \emph{round}\index{round}\index{convex!round}. Round convex sets have a hyperbolic behaviour. Non-round convex sets have common properties with ``the geometry of a normed plane''. An analogy should be made with symmetric spaces. We recall that rank-one symmetric spaces are Gromov-hyperbolic, and that higher rank symmetric space contain Euclidean planes.
}

\paragraph*{Acknowledgements}

\par{
The author thanks the anonymous referee for his useful comments and Athanase Papadopoulos for his careful and patience review. The author also wants to thanks Constantin Vernicos and Micka\"el Crampon for all the discussion they have around Hilbert geometry. This project of book is supported by the ANR Finsler.
}

\section{Generalities on the group $\Aut(\O)$}

The main goal of this chapter is to study the group $$\mathrm{Coll}^{\pm}(\O)= \Big\lbrace \g \in \PG \,|\, \g(\O) = \O \Big\rbrace$$ where $\O$ is a properly convex open set of $\PP^d$. The following fact is very basic and also very useful.

\begin{proposition}
The action of the group $\mathrm{Coll}^{\pm}(\O)$ on $\O$ is by isometries for the Hilbert distance. Consequently, the action of $\mathrm{Coll}^{\pm}(\O)$ on $\O$ is proper, so $\mathrm{Coll}^{\pm}(\O)$ is a closed subgroup of $\PG$ and so is a Lie group.
\end{proposition}

\par{
One may want to introduce the group $\mathrm{Isom}(\O)$ of isometries of $(\O,d_{\O})$ for the Hilbert metric. This group is rather mysterious. We postpone some remarks on the general knowledge on this group to the end of this chapter.
}
\\
\par{
Before going to a ``matrix study'' of the elements of $\mathrm{Coll}^{\pm}(\O)$, let us make a remark which allows us to see $\mathrm{Coll}^{\pm}(\O)$ as a subgroup of the group $\mathrm{SL}^{\pm}_{d+1}(\R)$ of linear transformations of $\R^{d+1}$ with determinant $1$ or $-1$. 
}
\\
\par{
We denote by $\PP:\R^{d+1}\smallsetminus \{ 0 \} \rightarrow \PP^d$ the canonical projection. Let $\O$ be a properly convex open subset of $\PP^d$. \emph{The cone above $\O$} \index{cone!above $\O$} is one of the two connected components $\C_{\O}$ of $\PP^{-1}(\O)$\footnote{This is a little abusive since there are two such cones, but nothing depends on this choice.}. We introduce the group $$\Aut^{\pm}(\O) = \Big\lbrace \g \in \mathrm{SL}^{\pm}_{d+1}(\R) \,|\, \g(\C_{\O})= \C_{\O} \Big\rbrace$$ rather than $\mathrm{Coll}^{\pm}(\O)$. The canonical morphism $\pi:\mathrm{SL}^{\pm}_{d+1}(\R) \rightarrow \PG$ is onto with kernel equal to $\{\pm 1\}$. But, the restriction of $\pi:\Aut^{\pm}(\O) \rightarrow \mathrm{Coll}^{\pm}(\O)$ is an isomorphism since elements of $\Aut^{\pm}(\O)$ preserve $\C_{\O}$.
}
\\
\par{
With this in mind, we will now concentrate on $\Aut(\O)$ which is the group of linear automorphism of $\R^{d+1}$ preserving $\C_{\O}$ and having determinant one:
}
$$\Aut(\O) = \Big\lbrace \g \in \ss \,|\, \g(\C_{\O})= \C_{\O} \Big\rbrace.$$
\par{
Sometimes it is very useful to look at the two-fold cover $\S^d$ of $\PP^d$. A projective way to define $\S^d$ is to consider it as the space of half-lines of $\R^{d+1}$.  The group of projective automorphism of $\S^d$ is the group $ \mathrm{SL}^{\pm}_{d+1}(\R)$. One advantage of $\S^d$ is that the definition of convexity is neater in $\S^d$ than in $\PP^d$.
}

\section{The automorphisms of a properly convex open set}

\subsection{The first lemma}

\begin{lemma}\label{lem_comp}
Every compact subgroup of $\Aut(\O)$ fixes a point of $\O$.
\end{lemma}

\begin{proof}
The proof of the lemma relies on the construction of a ``center of mass'' for every bounded subset of $\O$. This construction relies on the Vinberg convex hypersurface, see Lemma \ref{center}.
\end{proof}

\subsection{The matrix point of view}
\par{
Let $\g$ be an element of $\s{d+1}$. We denote by $\lambda_1(\g) \geqslant \lambda_2(\g) \geqslant \cdots \geqslant \lambda_{d+1}(\g)$ the moduli of the eigenvalues of $\g$ listed in decreasing order with multiplicity. The \emph{spectral radius}\index{spectral radius} $\rho_{\g}^+$ is by definition $\lambda_1(\g)$. We also define $\rho_{\g}^- =\lambda_{d+1}(\g) = \big(\rho_{\g^{-1}}^+\big)^{-1}$.
}
\\
\par{
An element $\g$ is \emph{semi-proximal}\index{semi-proximal}\index{proximal!semi} if $\rho_{\g}^+$ or $-\rho_{\g}^+$ is an eigenvalue of $\g$. An element $\g$ is \emph{positively semi-proximal}\index{positively semi-proximal} if $\rho_{\g}^+$ is an eigenvalue of $\g$.
}
\\
\par{
An element $\g$ is \emph{proximal}\index{proximal} if $\lambda_1(\g) > \lambda_2(\g)$. This implies that $\g$ is semi-proximal. An element $\g$ is \emph{positively proximal}\index{positively proximal}\index{proximal!positively} if $\g$ is positively semi-proximal and proximal. 
}
\\
\par{
Since, $\det(\g)  =1$, we remark that if $\rho_{\g}^+=1$ then $\rho_{\g}^-=1$ also, and $\lambda_i(\g)=1$ for all $i=1, \dots , d+1$.
}
\\
\par{
An element $\g$ is \emph{bi-``something''}\index{bi-} if $\g$ and $\g^{-1}$ are ``something''. The main use will be for \emph{biproximal}\index{biproximal}\index{proximal!bi} elements and \emph{positively biproximal}\index{positively biproximal} elements.
}
\\
\par{
Let $k=\R$ or $\mathbb{C}$. An element $\g$ is \emph{$k$-semi-simple}\index{semi-simple!$k$-semi-simple} if in a suitable  $k$-basis $\g$ is diagonal.
An element $\g$ is \emph{$\S^1$-semi-simple}\index{semi-simple!$\S^1$-semi-simple}\footnote{Usually, we say that $\g$ is hyperbolic when $\g$ is $\R$-semi-simple, and we say that $\g$ is elliptic when $\g$ is $\S^1$-semi-simple, but we shall not use this terminology because it will create a conflict with the next subsection.} if it is $\mathbb{C}$-semi-simple and all its eigenvalue are on the unit circle.
}
\\
\par{
An element $\g$ is \emph{unipotent}\index{unipotent} if $(\g-1)^{d+1}=0$. The \emph{power}\index{power} of a unipotent element is the smallest integer $k$ such that $(\g-1)^{k}=0$.
}
\\
\par{
For every element $\g \in \s{d+1}$ there exists a unique triple consisting of an $\R$-semi-simple element $\g_h$, an $\S^1$-semi-simple element $\g_e$ and a unipotent element $\g_u$ such that $\g= \g_h\g_e\g_u$ and these three elements commute with each other (see for example \textsection 4.3 of the book \cite{MR2158954} of Witte Morris or \cite{MR1064110} Chapter 3, Part 2).
}
\\
\par{
Let $\g \in \s{d+1}$ and $\lambda$ be an eigenvalue of $\g$. The \emph{power of $\lambda$}\index{power} is the size of the maximal Jordan block of $\g$ with eigenvalue $\lambda$. In another terms, the power of $\lambda$ is the multiplicity of $\lambda$ in the minimal polynomial of $\g$. Of course, when $\g$ is unipotent the power of 1 is exactly what we previously called the power of $\g$.
}

\begin{proposition}\label{classi_matrix}
Suppose the element $\g\in \ss$ preserves a properly convex open set $\O$. Then $\g$ is positively bi-semi-proximal. Moreover, if $\rho_{\g}^+ =1$ then its power is odd. Finally, the projective traces of the eigenspaces $\ker(\g - \rho_{\g}^+)$ and $\ker(\g - \rho_{\g}^-)$ meet the convex set $\overline{\O}$.
\end{proposition}

\begin{proof}
\par{
First, we have to prove that $\g$ is positively semi-proximal. This is exactly the content of Lemma 3.2 of \cite{MR2195260}. We give a rough proof.
}
\par{
Consider $V$ the subspace of $\R^{d+1}$ whose complexification $V_{\mathbb{C}}$ is the sum of the eigenspaces corresponding to the eigenvalues of modulus $\rho=\rho_{\g}^+$. If $x \in \O$, then a computation using the Jordan form of $\g$ shows that any accumulation point of the sequence $\g^n(x)$ belongs to $\overline{\O} \cap \PP(V)$.
}
\par{
Consider the relative interior $\O'$ of $\overline{\O} \cap \PP(V)$ which is a properly convex open subset of a vector space $V'$ and consider the restriction $\g'$ of $\g$ to $V'$. The element $\g'$ is semi-simple and all its eigenvalues have the same modulus $\rho$. Therefore, the element $\g'' = \g'/\rho$ is $\S^1$-semi-simple and preserves $\O'$. Since $\g''$ is $\S^1$-semi-simple, the group generated by $\g''$ is relatively compact. Lemma \ref{lem_comp} shows that $1$ is an eigenvalue of $\g''$ and that $\PP(\ker(\g''-1))$ meet $\O'$.
}
\\
\par{
The ``bi'' statement is obvious. Finally, we have to show that the Jordan multiplicity of $\rho$ in $\g$ is odd, when $\rho=1$. This is almost contained in Lemma 2.3 of \cite{MR2218481}. One has to study the action of $\Aut(\O)$ on $\S^d$, the 2-fold covering of $\PP^d$. Denote by $k$ the integer $\min \{ l \, |\, (\g - \rho)^l=0 \}$ and take any point $x \in \O$ which is not in $\ker(\g - \rho)^{k-1}$. A computation using a Jordan form of $\g$ shows that $\underset{n \to +\infty}{\lim} \g^n (x)  = - \underset{n \to -\infty}{\lim}\g^n (x)$ if $k$ is even, and $\underset{n \to +\infty}{\lim} \g^n (x)  = \underset{n \to -\infty}{\lim} \g^n (x)$ if $k$ is odd. Therefore $\g$ preserves a properly convex open set if and only if $k$ is odd.
}
\end{proof}


\subsection{The dynamical point of view}

\begin{definition}
Let $\g \in \Aut(\O)$. The \emph{translation length} of $\g$ on $\O$ is the quantity $\tau_{\O}(\g) = \underset{x \in \O}{\inf} \, d_{\O}(x,\g(x))$. We will say that $\g$ is:
\begin{enumerate}
\item \emph{elliptic}\index{elliptic} if $\tau_{\O}(\g) =0$ and the infimum is achieved;
\item \emph{parabolic}\index{parabolic} if $\tau_{\O}(\g) =0$ and the infimum is not achieved;
\item \emph{hyperbolic}\index{hyperbolic} if $\tau_{\O}(\g) > 0$ and the infimum is achieved;
\item \emph{quasi-hyperbolic}\index{quasi-hyperbolic} if $\tau_{\O}(\g) > 0$ and the infimum is not achieved.
\end{enumerate}
\end{definition}

A complete classification of automorphisms of properly convex open sets has been given in dimension 2 in \cite{MR1293655} and \cite{Marquis:2009kq}. There is a classification under the hypothesis that the convex set is round in \cite{Crampon:2012fk}. Finally, one can find a classification in the general context in \cite{Cooper:2011fk}. Our exposition is inspired by the last reference.

\subsubsection*{The faces of a properly convex open set.}

\par{
We present the notion of faces of a convex body. For more details and a different point of view, the reader can consult the book \cite{MR683612} at Section 1.5.
}
\\
\par{
Let $\O$ be a properly convex open subset of $\PP^d$. We introduce the following equivalence relation on $\overline{\O}$: $x \sim y$ when the segment $[x,y]$ can be extended beyond $x$ and $y$. The equivalence classes of $\sim$ are called \emph{open faces of $\overline{\O}$}\index{face!open face}; the closure of an open face is a \emph{face of $\overline{\O}$}\index{face}. The \emph{support}\index{support} of a face or of an open face is the smallest projective space containing it. The \emph{dimension of a face}\index{face!dimension} is the dimension of its support.}
\\
\par{
The behaviour of open faces and faces with respect to closure and relative interior is very nice thanks to convexity. More precisely, the interior of a face $F$ in its support (i.e its relative interior) is equal to the unique open face $f$ such that $\overline{f}= F$. Finally, one should remark that if $f$ is an open face of $\overline{\O}$ then $f$ is a properly convex open set in its support.
}
\\
\par{
The notion of face is important to describe the dynamics, more precisely to describe the set of attractive and repulsive fixed points of an automorphism. We give a family of examples. Consider the convex subset $\O$ of $\R^d$ formed by the points with strictly positive coordinates. This is a properly convex open set of $\PP^d$. The diagonal matrix $\g\in \ss$ given by $(\lambda_1,...,\lambda_{d+1})$ with $\lambda_i >0$, $\Pi \lambda_i =1$ and $\lambda_1 \leqslant \lambda_2 \leqslant \cdots \leqslant \lambda_{d+1}$ preserves $\O$. If $\lambda_1 > \lambda_2$ and $\lambda_d > \lambda_{d+1}$, then $\g$ has one attractive fixed point and one repulsive fixed point corresponding to the eigenlines of $\lambda_1$ and $\lambda_{d+1}$. Now, if $\lambda_1=\lambda_2 > \lambda_3$ and $\lambda_d > \lambda_{d+1}$, then $\g$ has a set of attractive fixed points and one repulsive fixed point. More precisely, the set of attractive fixed points of $\g$ is a segment included in $\dO$; it is even a face of $\O$. We can build a lot of examples with various sizes of attractive or repulsive fixed points with this convex $\O$. Each time the attractive set and the repulsive set will be a faces of $\O$. This is a general fact.
}
\subsubsection*{Horospheres in Hilbert geometry}

\paragraph*{The round case.}

Horospheres are very important metric objects in the study of the geometry of metric spaces. But horospheres are not easy to define in a general Hilbert geometry. We begin by a definition in the context of round Hilbert geometry.

Let $\O$ be a round convex subset of $\PP^d$. First, we define the \emph{Busemann function}\index{Busemann function} at a point $p \in \dO$. Let $x,y \in \O$. We set

$$\beta_{p}(x,y) = \underset{z \to p}{\lim} \,\,  d_{\O}(y,z)-d_{\O}(x,z).$$

To make the computation, compute first in any affine chart containing $\overline{\O}$ and then send $p$ to infinity. Take the supporting hyperplane $H$ of $\O$ at $p$ and do your computation in the affine chart $\A = \PP^d \smallsetminus H$. You get:

$$\beta_{p}(x,y) = \frac{1}{2} \ln \Bigg(\frac{|x-q_x|}{|y-q_y|} \Bigg)$$

\noindent where $q_x$ (resp. $q_y$) is the point $]p,x) \cap \dO$ (resp. $]p,y) \cap \dO$).

We recall that the \emph{horosphere}\index{horosphere} passing trough $x$ is the set of points $y$ of $\O$ such that $\beta_{p}(x,y)=0$. Therefore, in the affine chart $\A$, the horospheres are just the translates of the set $\dO \smallsetminus \{ p \}$ in the direction given by the line $p \in \PP^d$ which are in $\O$ (see Figure \ref{horo_fig}).

\begin{center}
\begin{figure}[h!]
  \centering
\includegraphics[width=5cm]{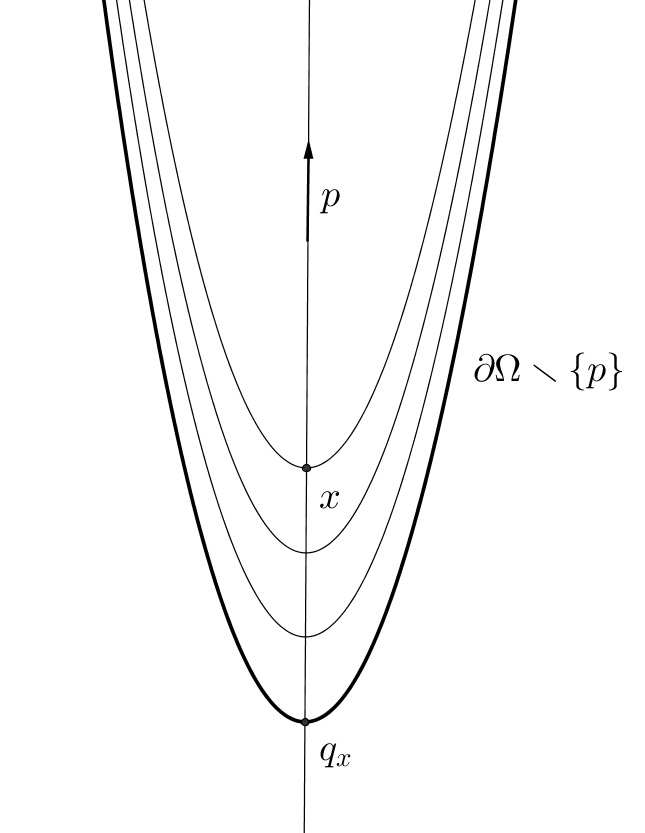}\
\caption{Algebraic horosphere for the ellipsoid \label{horo_fig}}
\end{figure}
\end{center}

Hence, we showed that in a round convex set the horospheres are round convex sets and have the same regularity than $\dO$.

\paragraph*{The general case: algebraic horospheres}

Walsh gives a complete description of the horofunction boundary of a Hilbert geometry in \cite{MR2456635}. We will not go that far in the study of horospheres, we will content ourselves with simpler objects.

The convergence of Busemann functions is no more true if the convex set is not round. Cooper, Long and Tillman introduced an alternative definition in \cite{Cooper:2011fk}. This time a horosphere will not be defined at a point but at a point $p$ together with a supporting hyperplane $H$ of $\O$ at $p$.

The \emph{algebraic horospheres based at $(p,H)$}\index{algebraic horosphere}\index{horosphere!algebraic horosphere} of a properly convex open set are the translates of $\dO \smallsetminus H$ in the affine chart $\PP^d \smallsetminus H$ in the direction\footnote{We insist on the fact that $p$ is a point of $\PP^d$ which is not in $\A$; therefore $p$ is also a direction in $\A$.} $p \in \PP^d$ which are in $\O$ (see Figure \ref{horo_tri}).

\begin{center}
\begin{figure}[h!]
  \centering
\includegraphics[width=5cm]{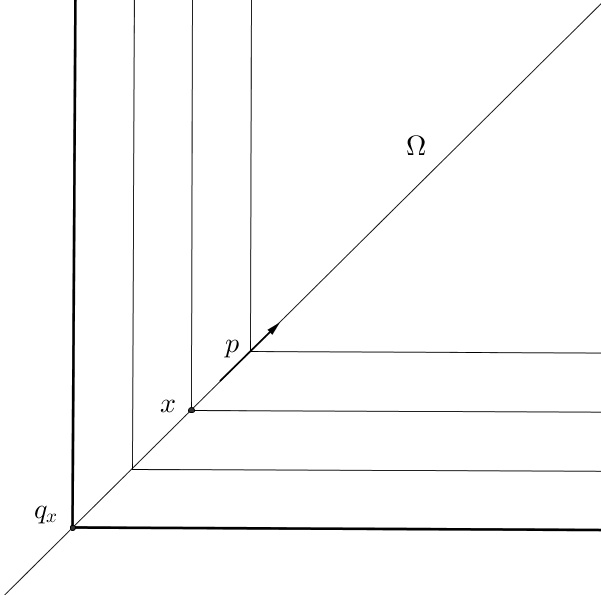}
\caption{Algebraic horosphere for the triangle \label{horo_tri}}
\end{figure}
\end{center}

Hence, we get two transverse foliations of $\O$, the one given by the straight line ending at $p$ and the one given by the algebraic horosphere based at $(p,H)$.

We need one more notion before stating and showing the classification.

\subsubsection*{Duality}\label{duality1}
\par{
We present the notion of duality for convex bodies. For more details and a different point of view, the reader can consult the books \cite{MR683612,MR1216521}.
}
\\
\par{
A convex cone of a real vector space is said to be \emph{sharp}\index{sharp}\index{convex cone!sharp} when it does not contain any affine line. Hence, there is a correspondence between open sharp convex cones of $\R^{d+1}$ and properly convex open subsets of $\PP^d$ via the natural projection.
}
\\
\par{
If $\C$ is a convex cone in a vector space $E$, then we define $\C^*=\{f \in E^* \,|\, \forall x \in \overline{\C}\smallsetminus \{ 0\}, \, f(x) > 0 \}$. We first remark that $\C^*$ is a convex open cone of $E^*$. Secondly, $\C^*$ is non-empty if and only if $\C$ is sharp and $\C^*$ is sharp if and only if $\C$ has non-empty interior. This operation defines an involution between the set of all sharp convex cones of $E$ and the  set of all sharp convex cones of $E^*$.
}
\\
\par{
If $\O$ is a properly convex open subset of $\PP(E)$ then we define \emph{the dual $\O^*$ of $\O$}\index{dual} by $\O^* = \PP(\C^*_{\O})$. This defines an involution between the set of all properly convex open subsets of $\PP(E)$ and the  set of all properly convex open subsets of $\PP(E^*)$.
}
\\
\par{
One can remark that $\O^*$ is a space of hyperplanes of $\PP^d$, hence $\O^*$ is also a space of affine charts since there is a correspondence between hyperplanes and affine charts. Namely, $\O^*$ can be identified with the space of affine charts $\A$ containing the closure $\overline{\O}$ of $\O$ or with the space of hyperplanes $H$ such that $H \cap \overline{\O}=\varnothing$. 
}
\\
\par{
An important proposition about duality for convex subsets is the following:
}
\begin{proposition}\label{prop_dual_classi_0}
Let $\O$ be a properly convex open set. Then $\O$ is strictly convex if and only if $\dO^*$ is $\C^1$.
\end{proposition}
\par{
We can explain this proposition in dimension 2. If $\dO$ is not $\C^1$ at a point $p \in \dO$ then the space of lines containing $p$ and not intersecting $\O$ is a segment in the boundary of $\O^*$, showing that $\O^*$ is not strictly convex.
}
\\
\par{
We consider the following set:
}
$$X^{\bullet}=\{(\O,x) \,|\, \O \textrm{ is a properly convex open set of } \PP^d \textrm{ and } x \in \O\}$$

\noindent endowed with the Hausdorff topology. The group $\PG$ acts naturally on $X^{\bullet}$. We denote by $\PP^{d *}$ the projective space $\PP((\R^{d+1})^*)$. We also consider its dual,
 
$$X^{\bullet *} = \{(\O,x) \,|\, \O \textrm{ is a properly convex open set of } \PP^{d *} \textrm{ and } x \in \O\}.$$ 

\noindent The following theorem needs some tools, so we postpone its proof to Part \ref{duality2} (Lemma \ref{centerofmass}) of this chapter. We think that this statement is natural and we hope the reader will feel the same.

\begin{lemma}\label{premier_dual}
There is a continuous bijection $^{\star}:X^{\bullet} \to X^{\bullet *}$ which associates to a pair $(\O,x)\in X^{\bullet}$ a pair $(\O^*,x^{\star}) \in  X^{\bullet *}$ where $\O^*$ is the dual of $\O$, and this map is $\PG$-equivariant.
\end{lemma}

\par{
We are now able to give a dynamical description of the automorphisms.
}

\subsubsection*{The classification in the general context}

\begin{proposition}
Let $\g \in \Aut(\O)$. The following are equivalent:
\begin{enumerate}
\item $\g$ is elliptic;
\item $\g$ fixes a point of $\O$;
\item $\g$ is $\S^1$-semi-simple.
\end{enumerate}
\end{proposition}

\begin{proof}
\par{
$1)\Leftrightarrow 2)$ This is the definition.

$2)\Rightarrow 3)$ Since $\g$ fixes a point $x$, $\g$ also fixes the hyperplane $x^{\star}$ dual to $x$ with respect to $\O$ (cf Lemma \ref{premier_dual}). Therefore, $\g$ is a linear transformation of the affine chart $\A = \PP^d \smallsetminus x^{\star}$ centered at $x$ which preserves the convex $\O$, hence $\g$ preserves the John ellipsoid\footnote{We recall that the John ellipsoid of a bounded convex $\O$ of a real vector space is the unique ellipsoid with maximal volume included in $\O$ and with the same center of mass.} of $\O$ (centered at $x$ in $\A$), so $\g$ is conjugate to an element in $\mathrm{SO}_d$. The conclusion follows.

$3)\Rightarrow 1)$ The closure $G$ of the group generated by $\g$ is compact, therefore every orbit of $\g$ is bounded, hence $\g$ is elliptic by Lemma \ref{lem_comp}.
}
\end{proof}

Given a point $p\in \dO$ and a supporting hyperplane $H$ at $p$ of $\dO$, one can define the group $\Aut(\O,H,p)$ of automorphisms of $\O$ which preserve $p$ and $H$. This group acts on the set of algebraic horosphere based at $(p,H)$ of $\O$. Since different algebraic horospheres correspond by translation in the direction $p$ in the chart $\PP^d \smallsetminus H$, we get a morphism $h:\Aut(\O,H,p) \rightarrow \R$ that measures horosphere displacement. The two following lemmas are left to the reader:

\begin{lemma}\label{dyna_horo}
Let $\g \in \Aut(\O,H,p)$. If $\rho_{\g}^+=1$ then $h(\g)=0$ and $\g$ preserves all the algebraic horospheres based at $(H,p)$ of $\O$. If $\rho_{\g}^+>1$ then $h(\g)\neq 0$.
\end{lemma}

\begin{lemma}[McMullen, Theorem 2.1 of \cite{MR1953192}]\label{petit_chiant}
Let $\g \in \Aut(\O)$. The translation length satisfies: $$\frac{1}{2} \ln \max(\rho_{\g}^+,1/\rho_{\g}^-,\rho_{\g}^+ / \rho_{\g}^-) \leqslant \tau_{\O}(\g) \leqslant \ln \max(\rho_{\g}^+,1/\rho_{\g}^-).$$ In particular, $\rho_{\g}^+=1$ if and only if $\g$ is elliptic or parabolic.
\end{lemma}

When a subset $A$ of $\PP^d$ is included in the closure $\overline{\O}$ of a properly convex open set $\O$, we can define its \emph{convex hull in $\O$}\index{convex hull}, i.e the smallest convex set of $\O$ containing $A$ in its closure. We will denote it by $\mathrm{Conv}(A)$. An element $\g \in \Aut(\O)$ is \emph{planar}\index{planar} when $\g$ is $\R$-semi-simple and has exactly two eigenvalues. In that case, the function $d_{\O}(x,\g(x))$ is constant on $\O$ and equal to $\frac{1}{2}\ln\big(\rho_{\g}^+/\rho_{\g}^-\big)$.

\begin{proposition}\label{classi_qhyp}
Let $\g \in \Aut(\O)$. The following are equivalent:
\begin{enumerate}
\item $\g$ is hyperbolic or quasi-hyperbolic.
\item $\rho_{\g}^+ >1$.
\item There exist two disjoint faces $F^+$ and $F^-$ of $\overline{\O}$ which are fixed and the action of $\g$ on the properly convex set $\O_{axe} = Conv(F^- \cup F^+)$ is planar.
\end{enumerate}
In the last case $\tau_{\O}(\g) = \frac{1}{2} \ln (\rho_{\g}^+ /\rho_{\g}^-)$. Moreover, $\g$ is quasi-hyperbolic if and only if $\O_{axe} \subset \dO$; otherwise $\g$ is hyperbolic. 
\end{proposition}

\begin{proof}
Lemma \ref{petit_chiant} shows that $1) \Rightarrow 2)$.

For $2) \Rightarrow 3)$ and $1)$, since $\rho_{\g}^+ >1$ we have $\rho_{\g}^- < 1$. So consider the projective spaces $E^+ = \PP(\ker(\g - \rho_{\g}^+))$ and $E^- = \PP(\ker(\g - \rho_{\g}^-))$, and let us set $F^+= \overline{\O} \cap E^+$ and $F^-= \overline{\O} \cap E^-$. Since the action of $\g$ on $\O$ is proper, the convex sets $F^+$ and $F^-$ are included in $\dO$. Moreover, a computation using the Jordan form of $\g$ shows that for any point $x \in \O$ the limit $\g^n x$ as $n$ tend to $+\infty$ (resp. $-\infty$) belongs to $F^+$ (resp. $F^-$), therefore $F^-,F^+$ are non-empty faces of $\overline{\O}$. Finally the restriction of $\g$ to the convex hull  $\O_{axe} = \mathrm{Conv}(F^- \cup F^+)$ is clearly planar, and so for every $x \in \O_{axe}$ we get $d_{\O}(x, \g(x)) =d_{\O_{axe}}(x, \g(x)) = \frac{1}{2} \ln (\rho_{\g}^+ /\rho_{\g}^-)$. Hence, $\tau_{\O}(\g)  \leqslant \frac{1}{2} \ln (\rho_{\g}^+ /\rho_{\g}^-)$ but  $\tau_{\O}(\g)  \geqslant \frac{1}{2} \ln (\rho_{\g}^+ /\rho_{\g}^-)$ thanks to Lemma \ref{petit_chiant}.

$3) \Rightarrow 2)$ Since the action of $\g$ on $\O_{axe}$ is planar we get $\rho^+_{\g} >1$.
\end{proof}

\begin{proposition}
Let $\g \in \Aut(\O)$. The following are equivalent:
\begin{enumerate}
\item $\g$ is parabolic.
\item $\rho_{\g}^+=1$ and $\g$ is not elliptic.
\item $\g_h=1$ and the power of $\g_u$ is odd and $\geqslant 3$.
\item $\g$ fixes every point of a face $F \subset \dO$ of $\overline{\O}$ and there is only one maximal face fixed by $\g$.
\item $\g$ preserves an algebraic horosphere and is not elliptic.
\end{enumerate}
\end{proposition}

\begin{proof}
$1) \Rightarrow 2)$ is a consequence of Lemma \ref{petit_chiant}.

$2) \Rightarrow 3)$ is a consequence of Proposition \ref{classi_matrix}.

For $3) \Rightarrow 4)$, denote by $k$ the power of $\g_u$ and consider the subspace $E= \PP(\mathrm{Im}(\g_u -1)^{k-1})$ of $\PP^d$. From the Jordan form of $\g$ we get that every point of $\O \smallsetminus \PP(\ker(\g_u -1)^{k-1})$ accumulates on a point of $F= \dO \cap E$; hence $F$ is a face of $\overline{\O}$. Moreover, every point of $F$ must be fixed by $\g$ since it is fixed by $\g_h$, and $\g$ and  $\g_h$ commutes. Finally, such an $F$ is unique since otherwise the action of $\g$ on $\mathrm{Conv}(F \cup F')$ would be planar and $\rho_{\g}^+ > 1$.

$4) \Rightarrow 5)$ Take any point $p$ in the relative interior of $F$. One can find a supporting hyperplane $H$ fixed by $\g$ that contains $F$. Since $F$ is unique we have $\rho_{\g}^+ = 1$, and one conclude with Lemma \ref{petit_chiant}.

$5) \Rightarrow 1)$ Lemma \ref{dyna_horo} shows that $\rho_{\g}^+=1$. Since $\g$ is not elliptic, Lemma \ref{petit_chiant} shows that $\g$ is parabolic. 
\end{proof}

\subsubsection*{The classification in the strictly convex case or $\C^1$ boundary case}

\begin{proposition}
Suppose the element $\g\in \ss$ preserves a properly convex open set $\O$. If $\O$ is strictly convex \textit{or} has $\C^1$ boundary then $\g$ is not quasi-hyperbolic. If $\g$ is hyperbolic then $\g$ is positively proximal and if $\g$ is parabolic then its Jordan block of maximal size is unique\footnote{We mean by this that $\g$ has only one Jordan block of size its power.}.
\end{proposition}

\begin{proof}
First, assume that $\O$ is strictly convex. If $\g$ is quasi-hyperbolic then by Proposition \ref{classi_qhyp}, one gets that $\O_{axe} \subset \dO$, hence $\O$ is not strictly-convex. Now, if $\g$ is hyperbolic and not proximal then it is an exercise to check that $\dO$ must contain a non-trivial segment. Analogously, if $\g$ is parabolic and its Jordan block of maximal size is not unique then $\dO$ must contain a non-trivial segment. Finally, if $\O$ has a $\C^1$ boundary then the dual $\O^*$ is strictly convex (Proposition \ref{prop_dual_classi_0}) and is preserved by $^t \g^{-1}$.
\end{proof}

\subsection{Examples}\label{examples}

\subsubsection{Ellipsoid}\label{presen_elli}

\par{
Consider the quadratic form $q(x)=x_1^2+ \cdots + x_d^2 - x_{d+1}^2$ on $\R^{d+1}$. The projective trace $\E_d$ of the cone of timelike vectors $\C_{d+1}=\{ x \in \R^{d+1} \, | \, q(x) < 0 \}$ is a properly convex open set. In a well chosen chart, $\E_d$ is a Euclidean ball. We call any image of $\E_d$ by an element of $\ss$ an \emph{ellipsoid}\index{ellipsoid}.
}
\\
\par{
The ellipsoid is the leading example of a round Hilbert geometry. The group $\Aut(\E_d)$ is the group $\SO$\footnote{More precisely it is the identity component of $\SO$.} and $(\E_d,d_{\E_d})$ is isometric to the real hyperbolic space of dimension $d$. In fact, $\E_d$ is the Beltrami-Cayley-Klein-projective model of the hyperbolic space.
}
\\
\par{
The elements of $\SO$ can be of three types: elliptic, hyperbolic or parabolic. The stabilizer of any point of $\E_d$ is a maximal compact subgroup of $\SO$, and it acts transitively on $\partial \E_d$. The stabilizer $P$ of any point $p$ of $\partial \E_d$ is an amenable subgroup of $\SO$ that acts transitively on $\E_d$. Moreover, $P$ splits as a semi-direct product of a compact group isomorphic to $\mathrm{SO}_{d-1}$ and a solvable subgroup $S$ which acts simply transitively on $\E_d$. And $S$ also splits as a semi-direct product $S= U \rtimes \R^*_+$ where $\R^*_+$ is the stabilizer of $p$ and any point $q\neq p$, $q \in \partial\E_d$. It is a group composed uniquely of hyperbolic elements, and $U$ is a unipotent subgroup composed only of parabolic elements. It is isomorphic to $\R^{d-1}$ and it acts simply transitively on each horosphere of $\E_d$. Finally, one can show that the Lie algebra $\mathfrak{u}$ of $U$ is conjugate to the following Lie algebra of $\mathfrak{sl}_{d+1}(\R)$:
}
$$
\left\{
(u_1,...,u_{d-1}) \in \R^{d-1} \, 
\left|
\left(
\begin{array}{ccccc}
0 & u_1 & \cdots & u_{d-1} & 0\\
  & 0    &    0         &   0          & u_1 \\
  &     &    \ddots         &       0      & \vdots \\
  &       &    &    0         & u_{d-1} \\
  &      &             &             & 0
\end{array}
\right)
\right.
\right\}
.$$

\par{
The case of the ellipsoid is not only the illuminating example of the world of round convex sets, its properties are also the main tool for studying groups acting on round convex sets.
}
\subsubsection{Cone over an ellipsoid.}\label{presen_conelli}

\par{
The cone of timelike vectors $\C_{d+1}= \{ x \in \R^{d+1} \, | \, q(x) < 0 \}$ is a sharp convex cone of $\R^{d+1}$, so it is also a properly convex open subset of $\PP^{d+1}$. It is not strictly convex nor with $\C^1$ boundary and its group of automorphisms is isomorphic to $\SO \times \R_+^*$. One should remark that the group $\Aut(\C_{d+1})$ is ``affine'' since every automorphism preserves the support of the ellipsoidal face of $\C_{d+1}$, hence the affine chart $\R^{d+1}$.
}
\\
\par{
This properly convex open set gives several counter-examples to statements about strictly convex open sets.
}

\subsubsection{Simplex.}\label{presen_simpl}

\par{
Consider an open simplex $S_d$ in $\R^d$. This is a polyhedron (and so a properly convex open subset of $\PP$) whose group of automorphisms is the semi-direct product of the group $D_d$ of diagonal matrices with positive entries and determinant one and the alternate group on the $d+1$ vertices of $S_d$.
}
\\
\par{
The group $D_d$ is isomorphic to $\R^d$ and it acts simply transitively on $S_d$. One should remark that $\Aut(S_d)$ does not preserves any affine chart of $\PP^d$ (i.e. it is not affine). This very simple properly convex open set is more interesting than it looks.
}
\\
\par{
One can remark that since every collineation of $\PP^d$ fixing $d+2$ points in generic position is trivial and every polyhedron of $\PP^d$ which is not a simplex has $d+2$ vertices in generic position, we get that the automorphism group of every polyhedron which is not a simplex is finite. Therefore the only polyhedron (as a properly convex open set) which can be of interest for the geometric group theorist is the simplex.
}

\subsubsection{The symmetric space of $\mathrm{SL}_m(\R)$.}

\par{
Consider the convex cone of symmetric positive definite matrices of $M_m(\R)$. It is a sharp convex cone in the vector space of symmetric matrices. We consider its trace $S_m^{++}$ on $\PP^d$ where $d=\frac{(m-1)(m+2)}{2}$. This is a properly convex open set which is not strictly convex nor with $\C^1$ boundary, if $m\geqslant 3$.
}
\\
\par{
The automorphism group of $S_m^{++}$ is $\s{m}$ via the representation $\rho_m(g) \cdot M =g M  {}^t \! g$. This representation offers plenty of examples of various conjugacy classes of matrices acting on a convex set.
}
\\
\par{
The interested reader can find a very nice description of $S_3^{++}$ and its relation to the triangle $S_2$ in \cite{MR1238518}.
}

\section{Examples of ``large'' groups acting on properly convex open subsets}

\par{
Before giving the examples we need a nice context and for that a notion of indecomposability. This is the goal of the next lines.
}
\\
\par{
Let us begin by a remark. Given a properly convex open set $\O$ of the projective space $\PP(E)$, the cone $\C_{\O}$ above $\O$ is a properly convex open set of the projective space $\PP(E\oplus \R)$. We will say that a properly convex open set $\O$ is a \emph{convex cone}\index{convex cone}\index{convex!cone} if there exists an affine chart such that $\O$ is a cone in this affine chart.
}
\\
\par{
The following definition is very natural: a sharp convex cone $\C$ of a vector space $E$ is \emph{decomposable}\index{decomposable}\index{convex cone!decomposable} if we can find a decomposition $E=E_1\oplus E_2$ of $E$ such that this decomposition induces a decomposition of $\C$ (i.e. $\C_i=E_i \cap \C$ and $\C=\C_1\times \C_2$). A sharp convex cone is \emph{indecomposable}\index{indecomposable}\index{convex cone!indecomposable} if it is not decomposable.
}
\\
\par{
We apply this definition to properly convex open sets. This need a little subtlety. A properly convex open set $\O$ is \emph{indecomposable}\index{indecomposable}\index{properly convex!indecomposable} if the cone $\C_{\O}$ above $\O$ is indecomposable and $\O$ is not a cone except if $\O$ is a segment.
}
\\
\par{
This definition suggests a definition of \emph{a product}\index{product}\index{properly convex!product} of two properly convex open sets which is not the Cartesian product. Namely, given two properly convex open sets $\O_1$ and $\O_2$ of the projective space $\PP(E_1)$ and $\PP(E_2)$, we define a new properly convex open set $\O_1 \otimes \O_2$ of the projective space $\PP(E_1 \times E_2)$ by the following formula: if $\C_i$ is the cone above $\O_i$ then $\O_1 \otimes \O_2 = \PP(\C_1 \times \C_2)$.
}
\\
\par{
It is important to note that if $\O_i$ is of dimension $d_i$ then $\O_1 \otimes \O_2$ is of dimension $d_1+d_2+1$. For example if the $\O_i$ are two segments then $\O_1 \otimes \O_2$ is a tetrahedron. Here is a more pragmatic way to see this product. Take two properly convex sets $\omega_i$ of a projective space $\PP(E)$ of disjoint support. The $\omega_i$ are not open but we assume that they are open in their supports. Then there exists an affine chart containing both $\omega_i$ and the convex hull, and such an affine chart\footnote{Of course, this does not depend of the choice of the affine chart containing the convex sets.} of $\omega_1 \cup \omega_2$ is isomorphic to $\omega_1 \otimes \omega_2$.
}
\\
\par{
Let us finish by an explanation of why the Cartesian product is not a good product for convex sets from our point of view. The Cartesian product is an affine notion and not a projective notion, namely the resulting convex set depends on the affine chart containing the convex sets. For example if the $\O_i$ are segments (which is a projective notion) then in an affine chart $\O_i$ can be a half-line or a segment. If they are both half-lines then the cartesian product is a triangle. But if they are both segments then the Cartesian product is a square which is not projectively equivalent to a triangle.
}
\\
\par{
We get that a properly convex open set is either indecomposable or a cone (which is not a segment) or it splits as a product for $\otimes$.
}
\\
\par{
One can easily remark that $\O_1 \otimes \O_2$ is never strictly convex nor with $\C^1$ boundary, that the automorphism group of  $\O_1 \otimes \O_2$ is by a canonical isomorphism the group $\Aut(\O_1) \times \Aut(\O_2) \times \R$ and that ``the group $\R$'' can be written in a good basis as a diagonal matrices group with diagonal entries: $ (\underbrace{\lambda,...,\lambda}_{d_1+1},  \underbrace{\mu,...,\mu}_{d_2+1})$ with $\lambda^{d_1+1}\mu^{d_2+1}=1$.
}

\subsection{Homogeneous properly convex open set}\label{sec_hom}

\par{
A properly convex open set $\O$ is said to be \emph{homogeneous} when its automorphism group acts transitively on it. Homogeneous convex sets have been classified in the sixties in \cite{MR0201575,MR0158414} by Vinberg. Rothaus gives an alternative construction of the list of all homogeneous convex sets in \cite{MR0202156}.
}
\\
\par{
If $\O_1$ and $\O_2$ are two homogeneous properly convex open sets, then the product $\O_1 \otimes \O_2$ is also homogeneous. In fact, one should also remark that the cone above $\O_1$  is a homogeneous properly convex open set with automorphism group $G_1 \times \R$. We will not spend time on homogeneous properly convex open sets because they do not give new examples of divisible or quasi-divisible convex. This is due to the following corollary of the classification:
}

\begin{proposition}
An homogeneous properly convex open set $\O$ is quasi-divisible if and only if it is symmetric. Moreover if $\O$ is indecomposable, it admits a compact and a finite-volume non-compact quotient.
\end{proposition}
\par{
A properly convex open set is \emph{symmetric}\index{symmetric}\index{properly convex!symmetric}\footnote{This is equivalent to the fact that : for every point $x \in \O$ there exists an \textit{isometry} $\g$ of $(\O,d_{\O})$ which fixes $x$ and whose differential at $x$ is $-Id$.} when it is homogeneous and self-dual\footnote{e.g $\O$ and $\O^*$ are isomorphic as properly convex open set.}.
}
\begin{proof}
\par{
We can assume that $\O$ is indecomposable. If $\O$ is symmetric then $G=\Aut(\O)$ is semi-simple and acts transitively and properly on $\O$. Theorem \ref{exis_lattice} shows that there exists a uniform lattice\footnote{A \emph{lattice} in a locally compact group $G$ is a discrete subgroup $\G$ such that the quotient $\Quotient{G}{\G}$ is of finite volume for the induced Haar measure. A lattice is \emph{uniform} when the quotient is compact.} $\G$ and a non-uniform lattice $\G'$ of $G$, therefore $\O$ admits a compact and a finite-volume non-compact quotient.
}
\\
\par{
If $\O$ is quasi-divisible then $G = \Aut(\O)$ admits a lattice, but a locally compact group which admits a lattice has to be unimodular (i.e its Haar measure is right-invariant and left-invariant). Therefore, the group $G$ is unimodular. The classification of Vinberg shows that the only case where $\O$ is homogeneous with $\Aut(\O)$ unimodular is when $\O$ is symmetric.
}
\end{proof}

\begin{theorem}[Borel - Harish-Chandra]\label{exis_lattice}
A semi-simple Lie group admits a uniform lattice and a non-uniform lattice.
\end{theorem}

\subsection{Symmetric properly convex open set}\label{sec_sym}

\par{
Symmetric properly convex open sets have been classified by Koecher in the sixties thanks to the classification of Jordan algebras (\cite{MR1718170}, \cite{MR1446489},\cite{MR0158414}). The classification of indecomposable symmetric properly convex open sets is given by the two following theorems:
}

\begin{theorem}
Let $\O$ be a symmetric properly convex open set of $\PP^d$ which is strictly convex. Then $\O$ is an ellipsoid. In other words, $\O$ is the symmetric space associated to $\so{d}$.
\end{theorem}

\par{
For $d=1$, we advise the reader that the usual name for an ellipsoid of dimension 1 is a segment.
}

\begin{theorem}
Let $\O$ be an indecomposable symmetric properly convex open subset of $\PP^d$ which is not strictly convex. Then $\O$ is the symmetric space associated to $\mathrm{SL}_m(\mathbb{K})$ where $\mathbb{K}=\R,\, \mathbb{C}, \, \mathbb{H}$ and $m \geqslant 3$ or to the exceptional Lie group $E_{6(-26)}$. 
\end{theorem}

\par{
We can give an explicit description of these symmetric properly convex open sets. For example for the symmetric properly convex open set associated to $\mathrm{SL}_m(\R)$ is the projective trace of the cone of positive definite symmetric matrices of size $m \times m$. For the other non-exceptional ones, just take the Hermitian or quaternionic Hermitian matrices.
}
\\
\par{
We remark that such convex sets do not exist in all dimensions and that the smallest one is of dimension 5. The real ones are of dimension $D_{\R}(d)=\frac{(d-1)(d+2)}{2} = 5, 9, 14, 20, 27, 35,...$, the complex ones are of real dimension $D_{\mathbb{C}}(d)=d^2-1= 8, 15, 24, 35, 48, 63,...$, the quaternionic ones $D_{\mathbb{H}}(d)=(2d+1)(d-1)= 14, 27, 44, 65, 90, 119, ...$, the exceptional one of dimension $26$.
}

\subsection{Spherical representations of semi-simple Lie groups}\label{spherical}

In \cite{MR565090}, Vinberg characterises the representation of a semi-simple real Lie group that preserves a properly convex open set. We use the classical notation. Let $K$ be a maximal compact subgroup of $G$ and let $P$ be a minimal parabolic subgroup of $G$. The characterisation is the following:

\begin{theorem}
Let $\rho:G \rightarrow \mathrm{SL}(V)$ be an irreducible representation of a semi-simple group $G$. Then the following are equivalent:
\begin{enumerate}
\item The group $\rho(G)$ preserves a properly convex open subset of $\PP(V)$.
\item The vector space $V^K$ of vectors fixed by $K$ is not zero.
\item The group $P$ preserves a half-line.
\end{enumerate}
Furthermore, if one of the previous assertions is true, then $\dim \big(V^K \big)=1$ and there exist two properly convex open sets $\O_{min}$ and $\O_{max}$ such that for any properly convex open set $\O$ preserved by $\rho$ we have $\O_{min}\subset \O \subset \O_{max}$. Such a representation is called a \emph{spherical} representation. 
\end{theorem}

The properly convex open set $\O_{min}$ is the convex hull of the limit set of $G$ and $\O_{max}$ is the dual convex to $\O_{min}$.

\subsection{Schottky groups}\label{schotkky}

\par{
The following subsection relies on the machinery of representation of real Lie groups. It can be skipped without consequence for the understanding of the rest of the chapter.
}
\\
\par{
Let $G$ be a semi-simple real linear connected Lie group and $\rho$ an irreducible representation of $G$ on a real vector space $V$ of finite dimension.
}
\\
\par{
One can ask the following question: \emph{When does there exist a Zariski-dense subgroup $\G$ of $G$ such that $\rho(\G)$ preserves a properly convex open set of $\PP(V)$~?
}
Benoist gives an answer to this question in \cite{MR1767272}. To present this theorem we need to introduce some vocabulary.
}
\\
\par{
Let $G$ be a linear algebraic semi-simple Lie group and $\rho:G \rightarrow \mathrm{V}$ an irreducible representation. We suppose that $G,\rho$ and $V$ are defined over the field $\R$. Let $B$ be a Borel subgroup of $G$ containing a maximal torus $T$ of $G$. Using $T$ one can decompose the representation $\rho$ into weight spaces $V^{\chi}$, $\chi \in X^*(T)=\Hom(T,\mathbb{C}^*)$. Let $\Pi(\rho)$ be the \emph{set of weights of the representation $\rho$}\index{representation!weight}, i.e $\Pi(\rho)= \{ \chi \in X^*(T) \, |\, V^{\chi} \neq 0 \}$. Take any order on $X^*(T)$ such that the roots of $G$ are positive on $B$. There is a unique maximal weight $\chi_0$\footnote{called \emph{the highest weight of $\rho$.}\index{representation!highest weight}} for this order and the corresponding weight space $V^{\chi_0}$ is a line. This line is the unique line of $V$ stabilized by $B$.
}
\\
\par{
We also need some ``real'' tools. The maximal torus $T$ contains a maximal $\R$-split torus $S$ and the Borel subgroup $B$ is contained in a minimal parabolic $\R$-subgroup $P$ of $G$. Using $S$, one can decompose the representation $\rho$ into weight spaces $V^{\lambda}$, $\lambda \in X^*(S)=\Hom(S,\mathbb{R}^*)$. Let $\Pi_{\R}(\rho)$ be the \emph{set of restricted weights of the representation $\rho$}\index{representation!restricted weight}, i.e $\Pi_{\R}(\rho)= \{ \lambda \in X^*(S) \, |\, V^{\lambda} \neq 0 \}$. The order on $X^*(T)$ induces an order on $X^*(S)$ such that the roots of $G$ are positive on $P$. There is a unique maximal restricted weight $\lambda_0$\footnote{called \emph{the highest restricted weight of $\rho$.}\index{representation!highest restricted weight}} for this order and the corresponding weight space $V^{\lambda_0}$ has no reason to be a line. The representation $\rho$ is \emph{proximal}\index{proximal}\index{representation!proximal} when $\dim(V^{\lambda_0})=1$.  We denote by $P_G$ the \emph{lattice of restricted weights}\index{representation!lattice of restricted weights} associated to $P$.
}
\\
\par{We have the following equivalences due to Abels, Margulis and So{\u\i}fer in \cite{MR1348303}:
\begin{enumerate}
\item The representation $\rho$ of $G$ is proximal.
\item The group $\rho(G_{\R})$ contains a proximal element.
\item Every element in the interior of a Weyl chamber of $S$ is mapped to a proximal element.
\item $V^{\chi_0}$ is stabilized by $P$.
\end{enumerate}

Furthermore, if one of the previous assertions is true, then the subset of $G_{\R}$ of all elements $g \in G$ such that $\rho(g)$ is proximal is Zariski-dense in $G$.
}
\\
\par{
A representation $\rho$ is \emph{orthogonal}\index{representation!orthogonal} (resp. \emph{symplectic}\index{representation!symplectic}) when $G$ preserves a non-degenerate symmetric (resp. antisymmetric) bilinear form on $V$. We say that two irreducible representations $\rho,\rho'$ are \emph{equal mod 2}\index{representation!equal mod 2} when the difference of their restricted highest weight is in $2P_G$.
}

\begin{theorem}[Benoist \cite{MR1767272}]\label{schottky}
Let $G$ be a semi-simple real linear connected Lie group  with finite center and $\rho$ an irreducible representation of $G$ on a real vector space $V$ of finite dimension.
\begin{enumerate}
\item $G$ contains a Zariski-dense subgroup $\G$ which preserves a properly convex open subset of $\PP(V)$ if and only if $\rho$ is proximal and not equal mod 2 to an irreducible proximal symplectic representation.
\item Every Zariski-dense subgroup $\G'$ of $G$ contains a Zariski-dense subgroup $\G$ which preserves a properly convex open subset of $\PP(V)$ if and only if $\rho$ is proximal and equal mod 2 to an irreducible proximal orthogonal representation.
\end{enumerate}
\end{theorem}

\begin{remark}
All subgroups $\G$ constructed by Benoist for this theorem are Schottky groups, so in particular they are discrete free groups.
\end{remark}

We give three examples. The interested reader can find more examples in the article of Benoist. The canonical representation of $\ss$ on $\R^{d+1}$satisfies 1) if and only if $d \geqslant 2$ and it never satisfies 2). The canonical representation of $\mathrm{SO}_{p,q}(\R)$ on $\R^{p+q}$ satisfies 1) and 2). The canonical representation of $\mathrm{SL}_m(\mathbb{C})$ on $\mathbb{C}^m$ never satisfies 1) nor 2).

\subsection{Strictly convex divisible and quasi-divisible convex sets}

\par{
A properly convex open set is \emph{divisible}\index{divisible} (resp. \emph{quasi-divisible}\index{quasi-divisible}) if the automorphism group of $\O$ contains a discrete subgroup $\G$ such that the quotient $\Quo$ is compact (resp. of finite volume).
}
\\
\par{
Figures \ref{dessin_nie1} and \ref{dessin_nie2} were made by Xin Nie, and they show tilings of 2-dimensional open convex bodies obtained by triangular Coxeter groups. In the first example, each tile is compact. In the second example each tile is of finite volume but not compact.
}
\begin{figure}[h!]
\centering
\includegraphics[width=4cm]{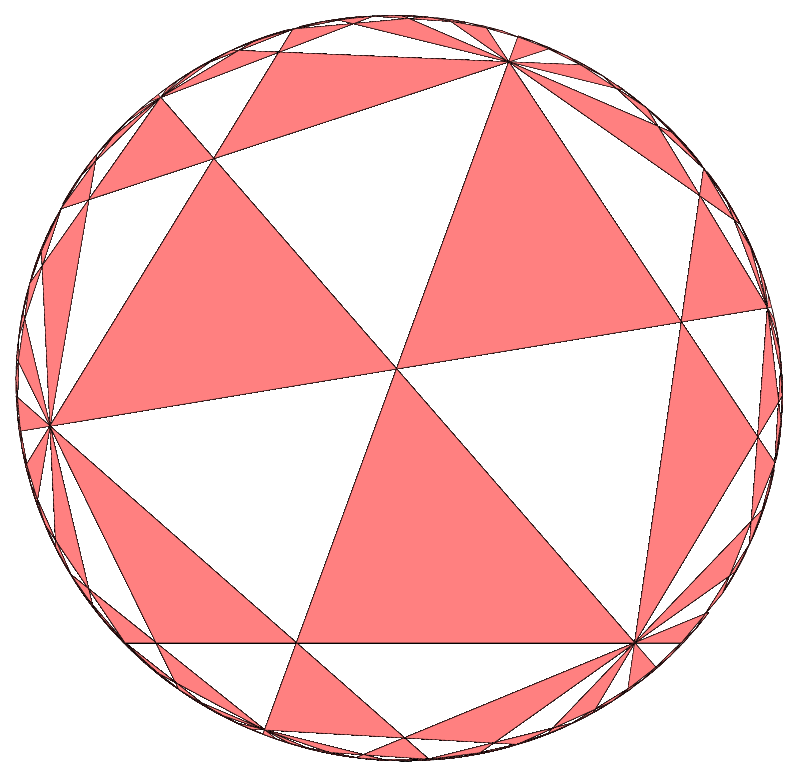}
\includegraphics[width=4cm]{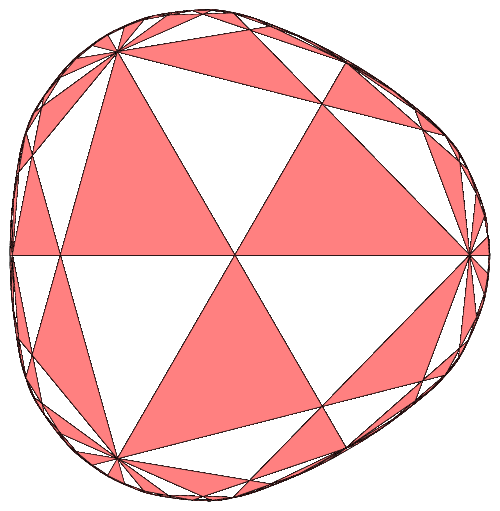}
\includegraphics[width=4cm]{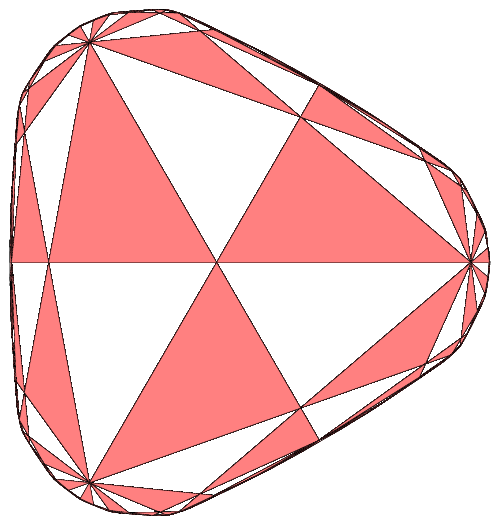}
\caption{Three divisible convex sets divided by the triangular Coxeter group $(3,3,7)$} \label{dessin_nie1}
\end{figure}
\begin{figure}[h!]
\centering
\includegraphics[width=4cm]{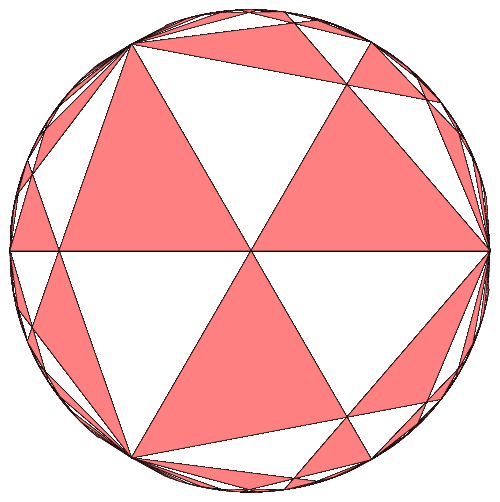}
\includegraphics[width=4cm]{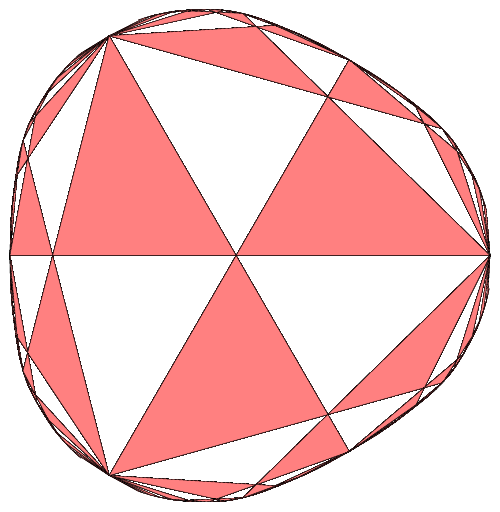}
\includegraphics[width=4cm]{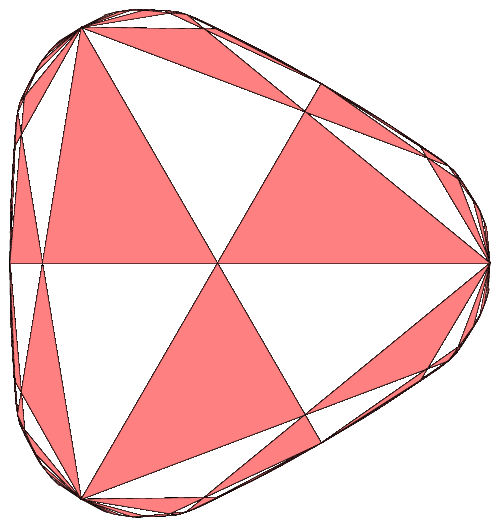}
\caption{Three quasi-divisible convex sets quasi-divided by the triangular Coxeter group $(3,3,\infty)$} \label{dessin_nie2}
\end{figure}

\subsubsection{Existence of non-trivial examples}

\begin{theorem}[Folklore]\label{exis_quasi}
In every dimension $d \geqslant 2$, there exists a divisible (resp. quasi-divisible but not divisible) convex set which is strictly convex and is not an ellipsoid.
\end{theorem}

Let us say a few words about the history of divisible and quasi-divisible convex sets:
\par{
The fact that the ellipsoid is divisible can be found at least in small dimensions in the work of Poincar\'e. The general case is due to Borel and Harish-Chandra (Theorem \ref{exis_lattice}). The fact that the only strictly convex open set which is homogeneous and divisible is the ellipsoid is a consequence of the work of Vinberg and Koecher as explained in Sections \ref{sec_sym} and \ref{sec_hom}.
}
\\
\par{
Kac and Vinberg showed using triangular Coxeter groups in \cite{MR0208470} that in dimension 2 there exist divisible convex sets which are not homogeneous, but they did not show that their examples were strictly convex (this is a consequence of a theorem of Benz\'ecri \cite{MR0124005} (Theorem \ref{thm_benz}) or the article \cite{MR0063115} of Kuiper). Johnson and Millson described in \cite{MR900823} a deformation\footnote{called bending and introduced by Thurston in \cite{CoursdeThurston} for quasi-fuschian groups.} of the projective structure of ``classical arithmetic'' hyperbolic manifolds. A previous theorem of Koszul in \cite{MR0239529} (Theorem \ref{thm_kosz}) showed that the deformed manifold is actually a \emph{convex} projective manifold if the deformation is small enough. Finally Benoist showed that the convex sets given by this deformation are strictly convex \cite{MR2094116} (Theorem \ref{thm_ghyp}) and that in fact the deformed structure is always convex even in the case of a big deformation \cite{MR2195260} (Theorem \ref{natur_stat}).
}
\\
\par{
In \cite{MR2740643} the author shows the theorem in dimension 2 by describing explicitly the moduli space of convex projective structures of finite volume on a surface. This explicit description is a small extension of Goldman's parametrisation in the compact case \cite{MR1053346}. In \cite{Marquis:2010fk}, the author shows the theorem in any dimension, using a ``bending construction'', but in that case, the convexity of the projective structure is obtained ``by hand'', i.e. without a theorem like Koszul theorem.
}
\\
\par{
Theorem \ref{exis_quasi} needs a security statement:
}
\begin{proposition}
If $\O$ is an indecomposable properly convex open set which has a compact and a finite-volume non-compact quotient, then $\O$ is a symmetric properly convex open set.
\end{proposition}

\begin{proof}
This proposition is a consequence of Theorem \ref{thm_zari1} which states that $\O$ is symmetric or that $\Aut(\O)$ is Zariski-dense in $\ss$. Therefore, either $\O$ is symmetric or $\Aut(\O)$ is discrete, since a Zariski-dense subgroup of a quasi-simple Lie group is either discrete or dense. Now, the case where $\Aut(\O)$ is discrete is excluded since the quotient of $\O$ by $\Aut(\O)$ would have to be compact and non-compact at the same time.
\end{proof}

\subsubsection{A non-existence result of exotic examples in small dimensions}

\begin{proposition}
Every group $\G$ dividing (resp. quasi-dividing) a strictly convex properly convex open set $\O$ of dimension 2 or 3 is a uniform lattice (resp. a lattice) of $\so{2}$ or $\so{3}$.
\end{proposition}

\begin{proof}
\par{
The articles \cite{Crampon:2012fk} and \cite{Cooper:2011fk} show that such a quotient $\Quo$ is the interior of a compact manifold. Moreover, such a manifold is aspherical since its universal cover is homeomorphic to $\R^d$.
}
\\
\par{
In dimension 2, $\Quo$ is homeomorphic to a compact surface with a finite number of punctures and this surface has negative Euler characteristic since $\O$ is strictly convex; therefore $\G$ is a lattice of $\so{2}$.
}
\\
\par{
In dimension 3, we distinguish the case of a compact quotient from the case of a non-compact quotient. Suppose $\Quo$ is of finite volume and non-compact, then $\Quo$ is the interior of an aspherical atoroidal compact 3-manifold with non-empty boundary, hence $\Quo$ is a Haken manifold and Thurston's hyperbolization Theorem of Haken manifolds implies that $\G$ is a non-uniform lattice of $\so{3}$. Finally, if $\Quo$ is compact then $\Quo$ is an aspherical atoroidal compact 3-manifold and $\G$ does not contain $\Z^2$ since $\G$ is Gromov-hyperbolic (Theorem \ref{thm_ghyp}). Hence Perelman's theorem on Thurston's geometrization conjecture shows that $\G$ is a uniform lattice of $\so{3}$.
}
\end{proof}

\subsubsection{An existence result of exotic examples in higher dimensions}

\begin{theorem}[Benoist $d=4$ \cite{MR2295544}, Kapovich $d\geqslant 4$ \cite{MR2350468}]$\,$\\
In every dimension $d \geqslant 4$, there exists a strictly convex divisible convex set which is not quasi-isometric to the hyperbolic space.
\end{theorem}

\par{
We just point out that the examples of Benoist are obtained thanks to a Coxeter group, and that Kapovich constructs a convex projective structure on some ``Gromov-Thurston manifold''. We recall that Gromov-Thurston manifolds are sequences of manifolds $M_n$ of dimension $d \geqslant 4$ such that none of them carries a Riemannian metric of constant curvature $-1$ but all of them carry a Riemannian metric of variable negative curvature such that the pinching constant converges to zero as $n$ goes to infinity (\cite{MR892185}).
}
\\
\par{
The following question is then natural and open:
}

\begin{oqu}
Does there exist in every dimension $d \geqslant 4$ a strictly convex quasi-divisible convex set which is not quasi-isometric to the hyperbolic space and which is not divisible ?
\end{oqu}

\subsection{Non strictly convex divisible and quasi-divisible convex sets}

\subsubsection{Intermission: Vey's theorem}

The following theorem describes the splitting of a divisible convex set into indecomposable pieces.

\begin{theorem}[Vey, \cite{MR0283720}]\label{theorem_vey}
Let $\G$ be a discrete subgroup of $\ss$ that divides a properly convex open set $\O$. Then:
\begin{enumerate}
\item There exists a decomposition $V_1 \oplus \cdots \oplus V_r$ of $\R^{d+1}$ such that $\G$ preserves this decomposition and for each $i=1,...,r$, the group $\G_i=\Stab_{\G}(V_i)$ acts strongly irreducibly on $V_i$.
\item The convex sets $\PP(V_i) \cap \overline{\O}$ are faces of $\O$. We denote by $\omega_1,...,\omega_r$ the corresponding open faces. We have $\O=\omega_1 \otimes \cdots \otimes \omega_r$.
\item Each $\omega_i$ is an indecomposable divisible convex set divided by $\G_i$.
\end{enumerate}
\end{theorem}

\begin{corollary}
Let $\G$ be a discrete subgroup of $\ss$ that divides a properly convex open set $\O$. Then $\O$ is indecomposable if and only if $\G$ is strongly irreducible.
\end{corollary}

This theorem leads to the following open question:

\begin{oqu}
Does Vey's theorem hold if $\G$ quasi-divides $\O$ ?
\end{oqu}

\subsubsection{A non-existence theorem of non-trivial examples in small dimensions}

\begin{theorem}[Benz\'ecri (divisible) \cite{MR0124005}, Marquis. (quasi-divisible) \cite{Marquis:2009kq}]\label{thm_benz}$\,$\\
In dimension 2, the only non-strictly convex divisible properly convex open set is the triangle, and there does not exist any quasi-divisible non-divisible non-strictly convex properly convex open set of dimension 2.
\end{theorem}

\begin{corollary}
In dimension 2, every indecomposable quasi-divisible convex open set is strictly convex with $\C^1$ boundary.
\end{corollary}

\subsubsection{An existence result}

\begin{theorem}[Benoist ($3\leqslant d \leqslant 7$, divisible) \cite{MR2218481}, Marquis. ($d =3$, quasi-divisible) \cite{MR2660566}]
In every dimension $3\leqslant d \leqslant 7$ (resp. $d=3$), there exists an indecomposable divisible (resp. quasi-divisible non-divisible) convex set which is not strictly-convex nor with $\C^1$ boundary.
\end{theorem}

We stress the fact that Benoist describes very precisely the structure of every divisible convex set in dimension 3 in \cite{MR2218481}. In particular, he shows that every segment in the boundary of $\O$ is in fact in a unique triangle $T$ such that $\partial T \subset \dO$ and such that every triangle in the boundary is stabilized by a virtually $\Z^2$-subgroup of $\G$. Moreover, every $\Z^2$-subgroup of $\G$ stabilizes such a triangle. Finally the projection of $T$ on the quotient $\Quo$ is a Klein bottle or a torus giving a geometric version of the Jaco-Shalen-Johannson decomposition of the quotient.

\subsubsection{A question in higher dimensions}

\begin{oqu}
Does there exist in every dimension $d \geqslant 4$ an indecomposable divisible (resp. a quasi-divisible not divisible) convex set which is not-strictly convex ?
\end{oqu}

\section{Convex hypersurfaces}

A very important tool in the study of groups acting on Hilbert geometries is the notion of a convex hypersurface. Informally, it consists in building the analogue of the hyperboloid associated to an ellipsoid for any properly convex open set.

\subsection{The theorem}
Let $X^{\bullet}=\{(\O,x) \,|\, \O \textrm{ is a properly convex open set of } \PP^d \textrm{ and } x \in \O\}$ endowed with the Hausdorff topology. The group $\PG$ acts naturally on $X^{\bullet}$.

\begin{theorem}[Vinberg, \cite{MR0158414}]\label{surf_vinberg}
Let $\O$ be a properly convex open subset of $\PP^d$, and $\C_{\O}$ the cone above $\O$. There exists a map $D_{\O}:\O \rightarrow \C_{\O}$ which defines a strictly convex analytic embedded hypersurface of $\R^{d+1}$ which is asymptotic to $\C_{\O}$ and this map is $\Aut(\O)$-equivariant. In fact, one can define this map as $X^{\bullet} \rightarrow \R^{d+1}$ so that it becomes $\PG$-equivariant.
\end{theorem}

A \emph{convex hypersurface of $\R^{d+1}$}\index{convex hypersurface} is an open subset of the boundary of a convex set of $\R^{d+1}$. A convex hypersurface $\Sigma$ is \emph{asymptotic}\index{convex hypersurface!asymptotic} to an open convex cone $\C$ containing $\Sigma$ if every affine half-line contained in the cone $\C$ intersects the hypersurface (see Figure \ref{asympto_fig}).

\begin{figure}[h!]
\centering
\includegraphics[width=4cm]{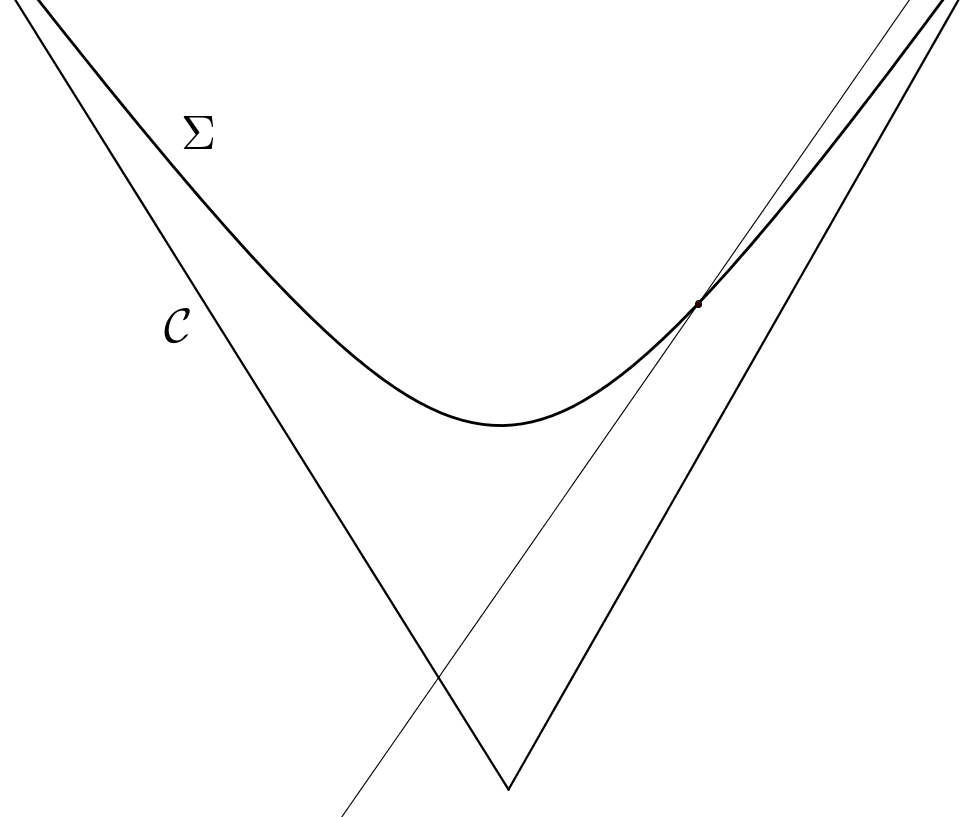}
\includegraphics[width=4cm]{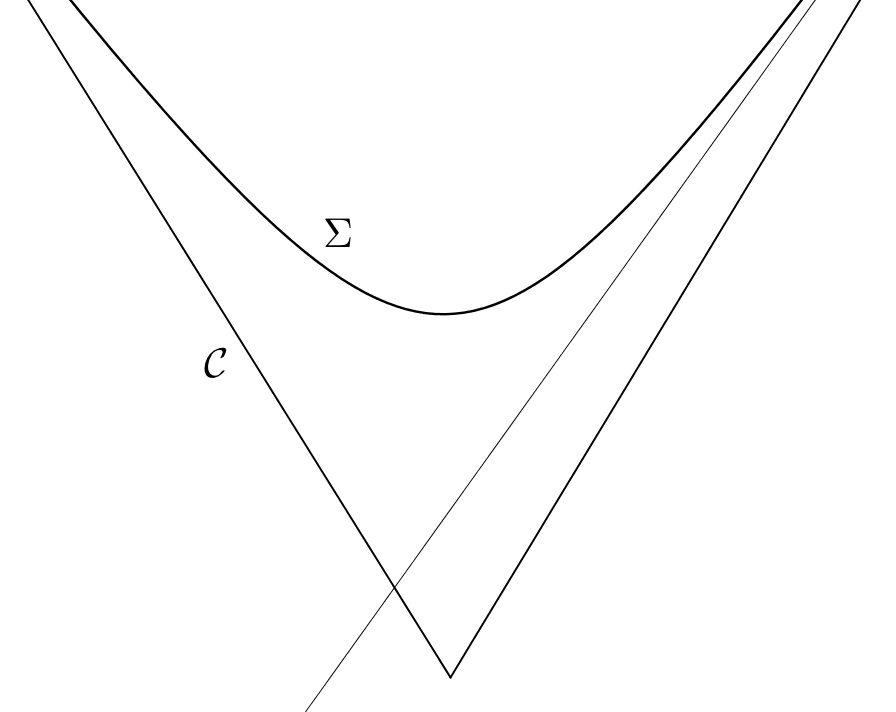}
\caption{An asymptotic and a non-asymptotic convex hypersurface} \label{asympto_fig}
\end{figure}

\begin{proof}
Consider the map $\varphi:\C_{\O} \rightarrow \R^*_+$ given by $\varphi(x)=\int_{\C_{\O}^*} e^{-f(x)}df$. We leave it as an exercise that the level set of this map gives the intended hypersurface.
\end{proof}

\subsection{Consequences}

\subsubsection*{Existence of centers of mass}

\begin{lemma}\label{center}
Let $B$ be a bounded part of a properly convex open set. There exists a point $x_B \in \mathrm{Conv}((B)$ such that for every $\g \in \Aut(\O)$, if $\g(B)=B$ then $\g(x_B)=x_B$.
\end{lemma}

\begin{proof}
Denote by $D:\O \rightarrow \Sigma$ the convex hypersurface given by Theorem \ref{surf_vinberg}. Consider $C=\overline{\mathrm{Conv}(D(B))}$. It is a bounded convex part of $\R^{d+1}$. The action of $\Aut(\O)$ on $\R^{d+1}$ is linear, so the projection on $\O$ of the center of mass of $C$ is the point $x_B$ we are looking for. 
\end{proof}

\subsubsection*{Existence of an invariant Riemannian metric}\label{metric_riema}

\begin{lemma}
Let $\O$ be a properly convex open set. There exists on $\O$ a Riemannian metric which is $\Aut(\O)$-invariant. In fact, one can define this metric from $X^{\bullet}$ so that it becomes $\PG$-equivariant.
\end{lemma}

\begin{proof}
Since the hypersurface of Theorem \ref{surf_vinberg} is strictly convex, its Hessian is definite positive at every point, therefore it defines an invariant analytic Riemannian metric.
\end{proof}

\subsubsection*{Existence of a convex locally finite fundamental domain}

\begin{lemma}\label{bisector}
Let $\O$ be a properly convex open set. There is a map $H$ from $\O\times \O$ to the space of hyperplanes intersecting $\O$ such that for every $x,y \in \O$ the hyperplane $H(x,y)$ separates $x$ from $y$ and $H$ is $\Aut(\O)$-equivariant, and such that if $y_n \to p \in \dO$ then $H(x,y_n)$ tends to a supporting hyperplane at $p$.
\end{lemma}

\begin{proof}
Consider the \emph{affine} tangent hyperplanes $H_x$ and $H_y$ at the point $D(x)$ and $D(y)$ of the convex hypersurface $\Sigma=D(\O)$. Set $H(x,y)=\PP(\mathrm{Vect}(H_x \cap H_y))$. It is an exercise to check that $H(x,y)$ does the job (see Figure \ref{bisector_fig}).
\begin{figure}[h!]
\centering
\includegraphics[width=7cm]{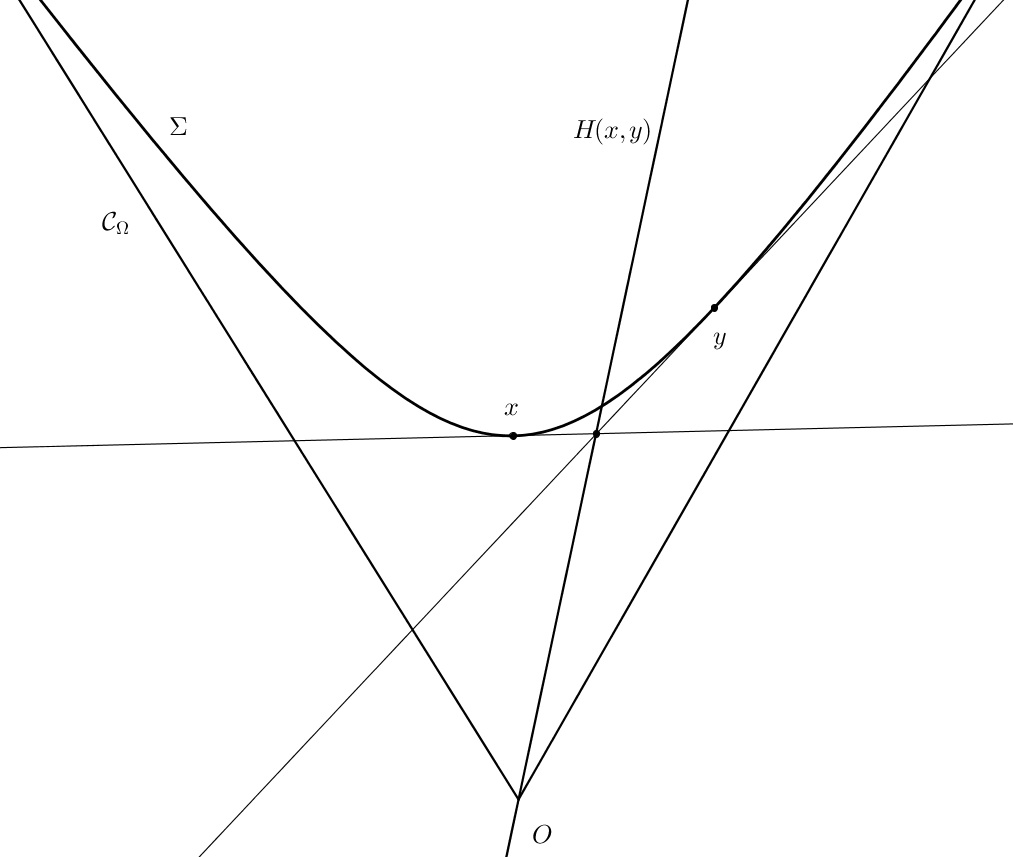}
\caption{Bisector} \label{bisector_fig}
\end{figure}
\end{proof}

\par{
The existence of convex locally finite fundamental domains for actions of discrete groups is very useful. For the hyperbolic space, the existence of fundamental domains is due to Dirichlet. A direct application of the Dirichlet techniques does not work in Hilbert Geometry since the bisector of two points is no longer a hyperplane. Indeed, in the case of Hilbert geometry the two connected component given by a bisector have no reason to be convex.
}
\\
\par{
Nevertheless, thanks to Lemma \ref{bisector}, Lee shows:
}

\begin{theorem}[Lee \cite{MR2712298} or \cite{Lee:2007fk}]
Let $\G$ be a discrete group of $\ss$ acting on a properly convex open set $\O$. There exists a locally finite\footnote{A fundamental domain is \emph{locally finite}\index{locally finite} when each compact subset of $\O$ intersects only a finite number of translates of the fundamental domain.} convex fundamental domain for the action of $\G$ on $\O$.
\end{theorem}

\par{
The method of Lee relies on the existence and geometry of affine spheres (cf. next paragraph). The reader can also find a proof which relies only on Vinberg's hypersurface in \cite{Marquis:2009kq}.
 }
 
\subsection{An alternative construction}
 
There is an another construction of similar convex hypersurfaces.

\begin{theorem}[Cheng-Yau, Calabi-Nirenberg]\label{surf_yau}
Let $\O$ be a properly convex open set of $\PP^d$. There exists a map $D_{\O}':\O \rightarrow \C_{\O}$ which defines a strictly convex embedded hyperbolic affine sphere\footnote{We refer to the survey \cite{MR2743442} of Loftin for a definition of an affine sphere and for references.} of $\R^{d+1}$ which is asymptotic to $\C_{\O}$ and this map is $\Aut(\O)$-equivariant. In fact, one can define this map from $X^{\bullet}$ to $\R^{d+1}$ so that it becomes $\PG$-equivariant.
\end{theorem}

The main application of this theorem is Theorem \ref{complex_moduli}.

\section{Zariski Closure}


\subsection{Definitions of Proximality}

\par{
A subgroup $\G$ of $\ss$ is \emph{proximal}\index{proximal}\index{proximal!group} if it contains a proximal element. 
}
\\
\par{
The \emph{action of a group $\G$ on $\PP^d$ is proximal}\index{proximal!action} if for every $x,y \in \PP^d$, there exists a $\g \in \G$ such that $\g(x)$ and $\g(y)$ are arbitrarily close.
}
\\
\par{
We give the definition of a \emph{proximal representation} of a semi-simple real Lie group in Subsection \ref{schotkky}.
}
\par{
These three definitions of proximality have the following link. When $\G$ is strongly irreducible,  the action of $\G$ on $\PP^d$ is proximal if and only if $\G$ is proximal. Moreover, if $G$ is the identity component of the Zariski closure of $\G$, then the representation $\rho:G \rightarrow \ss$ is proximal if and only if $G$ is proximal if and only if $\G$ is proximal.
}


\subsection{Positive proximality}

The following theorem characterizes irreducible subgroups of $\s{d+1}$ which preserve a properly convex open set. 

A proximal group $\G$ is \emph{positively proximal} when every proximal element is positively proximal.

\begin{theorem}[Benoist \cite{MR1767272}]\label{theo_pos_prox}
A strongly irreducible subgroup of $\s{d+1}$ preserves a properly convex open set if and only if $\G$ is positively biproximal. 
\end{theorem}

If $G$ is the identity component of the Zariski closure of $\G$ then $G$ has no reason to be positively biproximal. The group $G$ is positively biproximal if and only if the representation $\rho:\G \rightarrow \ss$ is spherical (cf \textsection \ref{spherical}). This phenomenon is not exceptional: just take any divisible convex set $\O$ which is not an ellipsoid. Any group dividing it is Zariski dense in $\ss$ by Theorem \ref{thm_zari1} below and $\ss$ is not positively proximal.

Finally, we stress on the following fact about positively proximal groups.

\begin{theorem}
Suppose $\G$ is a strongly irreducible group which preserves a properly convex open set. Then there exists a unique closed $\G$-invariant minimal subset $\LG$ of $\PP^d$; hence there exist two properly convex open sets $\O_{min}$ and $\O_{max}$ preserved by $\G$ such that for every properly convex open set preserved by $\G$, we have $\O_{min} \subset \O \subset \O_{max}$.
\end{theorem}
\par{
Therefore, if we start with a strongly irreducible group $\G$ preserving a properly convex open set, by taking its Zariski closure we get a reductive\footnote{A linear Lie group is \emph{reductive}\index{reductive} when it does not contain any non-trivial normal unipotent subgroup.} group $G$. We then get an irreducible representation $\rho:G \rightarrow \s{d+1}$ which is proximal. Since we assume that $\G \leqslant \ss$, we get that $G$ is in fact semi-simple\footnote{Indeed, since $G$ is reductive, we just need to show that the center of $G$ is discrete. Take any element $g$ in the center of $G$; $g$ has to preserve all the eigenspaces of all the proximal elements of $G$, hence $g$ is a homothety, so $g=\pm 1$ since $G \leqslant \ss$.}. The irreducible representations of a semi-simple group are completely classified. The next question is: what can we say about this representation ?
}
\\
\par{
Theorem \ref{schottky} gives a complete answer to this question. But we can say more in the case of a finite-covolume action.
}

\subsection{Cocompact and finite-covolume case}

The following theorem of Benoist completely describes the Zariski closure of a group $\G$ dividing a properly convex open set.

\begin{theorem}[Benoist \cite{MR2010735}]\label{thm_zari1}
Suppose that $\G$ divides an indecomposable properly convex open set which is not homogeneous. Then $\G$ is Zariski-dense.
\end{theorem}

The following question is open:

\begin{oqu}\label{qu_zar}
Suppose that $\G$ quasi-divides an indecomposable convex set $\O$ which is not homogeneous. Is $\G$ Zariski-dense ?
\end{oqu}

The following theorem answers the question \ref{qu_zar}for quasi-divisible convex sets in the strictly convex case.

\begin{theorem}[Crampon-Marquis. \cite{Crampon:2012fk}]\label{thm_zari2}
Suppose that $\G$ quasi-divides a strictly convex open set which is not an ellispoid. Then $\G$ is Zariski-dense.
\end{theorem}

For results of this kind in the context of geometrically finite actions the reader is invited to read \cite{Crampon:2012fk}.

\subsection{A sketch proof when $\O$ is strictly convex}

We only sketch the proof when the convex is strictly convex. The actual techniques for the non-strictly convex case are different but the hypothesis of strict convexity already gives a nice feeling of the proof.

\begin{lemma}[Vey, \cite{MR0283720}]\label{lemma_vey}
Let $\G$ be a discrete subgroup of $\ss$ that divides a properly convex open set $\O$. Then, for every $x \in \O$, the convex hull of the orbit of $x$ is $\O$, i.e $\mathrm{Conv}(\G \cdot x) = \O$. In particular, if $\O$ is strictly convex then $\LG= \dO$.
\end{lemma}

\begin{proof}
The action is cocompact so there exists a number $R_0>0$ such that given any point $x \in \O$, the projection of the ball $B_x(R_0)$ of radius $R_0$ is $\Quo$; in other words the ball $B_x(R_0)$ meets every orbit. 

Let $p$ be an extremal point of $\dO$. We are going to show that $p \in \overline{\G \cdot x}$, and this will prove the first part of the lemma. Suppose that $p \notin \overline{\G \cdot x}$. Then there exists a neighbourhood $\U$ of $p$ such that $\U \cap \G \cdot x = \varnothing$; but since $p$ is extremal, the set $\U\cap \O$ contains balls of $(\O,d_{\O})$ of arbitrary size, contradicting the first paragraph.

Since $\dO$ is closed and $\G$-invariant we always have $\LG \subset \dO$. Now if $\O$ is strictly convex, suppose that $\LG \subsetneq \dO$. Then since $\O$ is strictly convex we get that $\mathrm{Conv}(\LG) \subsetneq \O$; in particular the convex hull of the orbit of any point of $\mathrm{Conv}(\LG)$ is not all of $\O$, contradicting the first part of this lemma. Hence, $\LG = \dO$ when $\O$ is strictly convex.
\end{proof}

\begin{proof}[Proof of \ref{thm_zari1} in the case where $\O$ is strictly convex]
Let $G$ be the identity component of the Zariski closure of $\G$. Since $\O$ is strictly convex, $\LG=\dO$ by Lemma \ref{lemma_vey}, $\G$ is strongly irreducible by Theorem \ref{theorem_vey} and positively proximal by Theorem \ref{theo_pos_prox}; hence the group $G$ is semi-simple and the representation $\rho:G \rightarrow \ss$ is irreducible and proximal.

Consider the limit set $\Lambda_{G}$. The limit set is an orbit of $G$ because the action of $G$ on $\PP^d$ is by projective transformations hence every orbit is open in its closure. Hence the limit set is the unique closed orbit of $G$ acting on $\PP^d$. It is the orbit of the line of highest restricted weight. Hence, $\Lambda_G \supset \LG=\partial \O$. There are two possibilities, $\Lambda_G=\dO$ or $\Lambda_G=\PP^d$.

In the first case, the maximal compact subgroup $K$ of $G$ acts also transitively on $\Lambda_G=\dO$. But $K$ fixes a point $x$ of $\O$ and a point $x^{\star}$ of $\O^*$ since $K$ is compact, hence $K$ preserves the John ellipsoid $\E$ of $\O$ centered at $x$ in the affine chart $\PP^d\smallsetminus x^{\star}$. As $K$ acts transitively on $\dO$, we get that $\O=\E$.

In the second case, the lemma below shows that either $G=\ss$ or $G=\mathrm{Sp}_{2d}(\R)$ and $\G$ preserves a symplectic form. But a group which preserves a properly convex open set cannot preserve a symplectic form (Corollary 3.5 of \cite{MR1767272}).
\end{proof}

\begin{lemma}[Benoist, Lemma 3.9 of \cite{MR1767272}]
Let $G$ be a linear semi-simple connected Lie group and $(\rho,V)$ a faithful irreducible and proximal representation. The action of $G$ on $\PP(V)$ is transitive if and only if
\begin{enumerate}
\item $G=\ss$ and $V=\R^{d+1}$ with $d \geqslant 1$ or
\item $G=\mathrm{Sp}_{2d}(\R)$ and $V=\R^{2d}$ with $d \geqslant 2$.
\end{enumerate}
\end{lemma}

If one reads carefully the proof, one should remark that we only used the hypothesis ``$\O$ is strictly convex'' to get that $\LG=\dO$. So, in fact we have shown that:

\begin{theorem}
Let $\G$ be a discrete group of $\ss$ that preserves a properly convex open set $\O$. If $\G$ acts minimally on $\dO$ then $\G$ is Zariski-dense or $\O$ is an ellipsoid.
\end{theorem}

This leads to the following question:

\begin{oqu}
Let $\G$ be a discrete group of $\ss$ that divides (or quasi-divides) an indecomposable properly convex open set $\O$ which is not homogeneous. Does $\G$ act minimally on $\dO$ ?
\end{oqu}

If the answer to this question is yes, then one gets an alternative proof of Theorem \ref{thm_zari1}. We remark that Benoist answers the last question in dimension 3 in \cite{MR2218481}. Namely, he shows that the action of any group $\G$ dividing an indecomposable properly convex open set is minimal on $\dO$.

\section{Gromov-hyperbolicity}

The notion of Gromov-hyperbolicity is a very powerful tool in geometric group theory and metric geometry. The goal of this part is to catch the link between Gromov-hyperbolicity and roundness of convex bodies.

A proper geodesic metric space $X$ is \emph{Gromov-hyperbolic}\index{Gromov-hyperbolic}\index{geodesic metric space!Gromov-hyperbolic} when there exists a number $\delta$ such that given any three points $x,y,z \in X$, and given any geodesics $[x,y]$, $[y,z]$ and $[z,x]$, the geodesic $[x,y]$ is included in the $\delta$-neighbourhood of $[y,z]\cup [z,x]$.

A group $\G$ of finite type is \emph{Gromov-hyperbolic} if its Cayley graph given by one of its finite generating sets\footnote{This property does not depend on the choice of the generating set.} is Gromov-hyperbolic for the word metric.

\begin{theorem}[Benoist \cite{MR2094116}]\label{thm_ghyp}
Let $\O$ be a divisible convex set divided by a group $\G$. Then the following are equivalent:
\begin{enumerate}
\item The metric space $(\O,d_{\O})$ is Gromov-hyperbolic.
\item The convex $\O$ is strictly convex.
\item The boundary of $\O$ is $\C^1$.
\item The group $\G$ is Gromov-hyperbolic.
\end{enumerate}
\end{theorem}

Note that a similar statement is true by \cite{Cooper:2011fk} in the case of quasi-divisible convex sets. There is also a statement of this kind in \cite{Choi:2010fk} for non-compact quotients. And, finally there is a weaker statement of this kind for geometrically finite actions in \cite{Crampon:2012fk}. We will not review these results.

We shall present a rough proof of this theorem. 

\subsection{The first step}

\begin{proposition}[Benoist \cite{MR2094116} - Karlsson-Noskov \cite{MR1923418}]\label{prop_ghyp_facile}
Let $\O$ be a properly convex open set. If the metric space $(\O,d_{\O})$ is Gromov-hyperbolic then the convex $\O$ is strictly convex with $\C^1$ boundary.
\end{proposition}

\begin{proof}
\par{
Suppose $\O$ is not strictly convex, and take a maximal non trivial segment $s \subset \dO$. Choose a plane $\Pi$ containing $s$ and intersecting $\O$, choose a sequence of points $x_n$ and $y_n$ in $\O\cap  \Pi$ converging to the different endpoints of $s$, and finally take any point $z \in \O\cap  \Pi$. It is an exercise to show that $\underset{u \in [x_n, y_n]}{\sup} d_{\O}(u,[x_n, z] \cup [z, y_n]) \to \infty$, which shows that $\O$ is not Gromov-hyperbolic (see Figure \ref{gh_strict_fig})
}
\begin{center}
\begin{figure}[h!]
  \centering
\includegraphics[width=7.5cm]{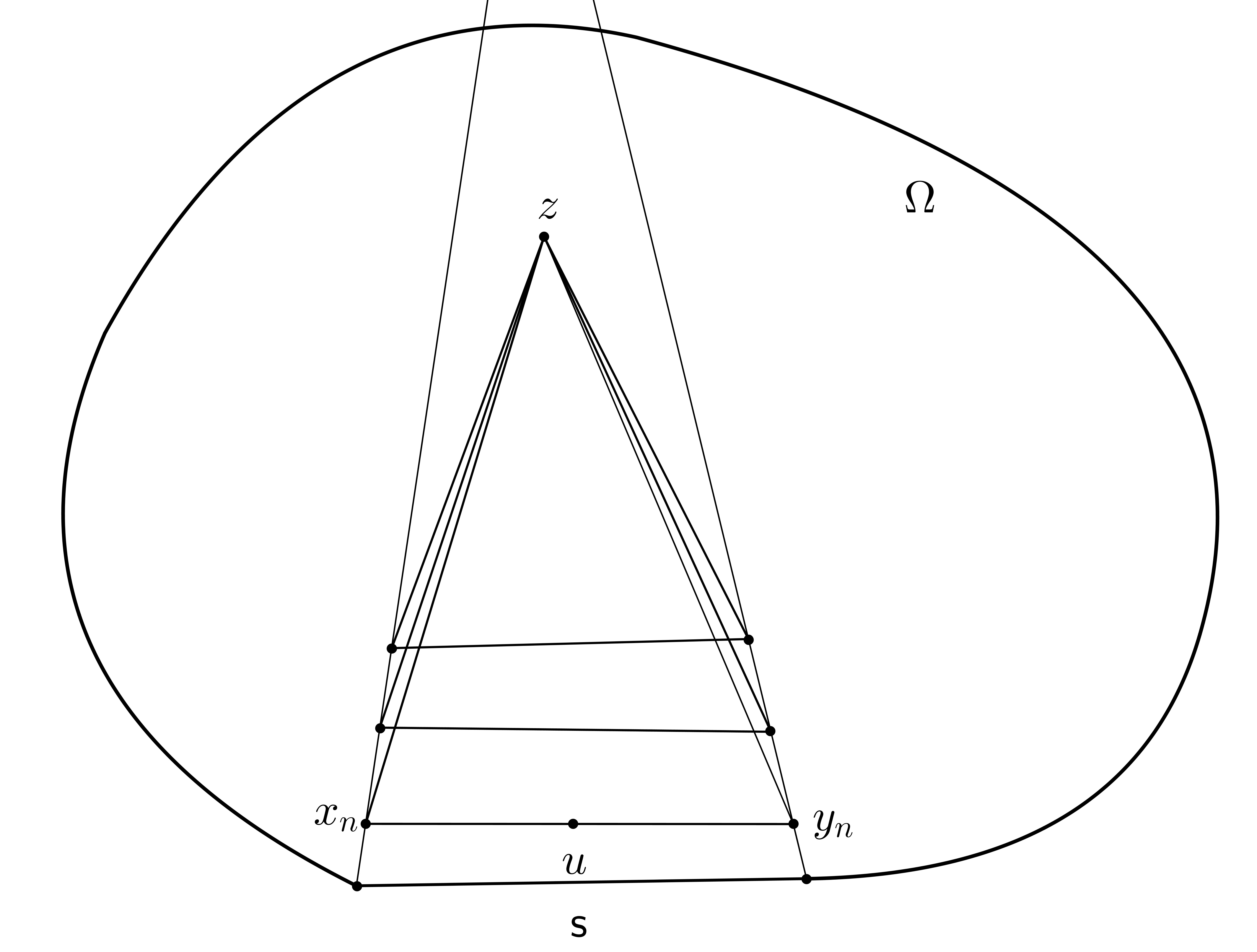}\
\caption{Strict convexity proof}
\label{gh_strict_fig}
\end{figure}
\end{center}

\par{
Suppose that $\dO$ is not $\C^1$ but Gromov-hyperbolic. Then there exist a point $z \in \dO$, a plane $\Pi \ni z$ such that $\O \cap \Pi \neq \varnothing$, a triangle $T$ of $\Pi$ containing $\O \cap \Pi$ such that $z$ is a vertex of $T$ and the two segments of $T$ issuing from $z$ are tangent to $\O \cap \Pi $ at $z$ (see Figure \ref{gh_C1_fig}).
}

\begin{center}
\begin{figure}[h!]
  \centering
\includegraphics[width=7cm]{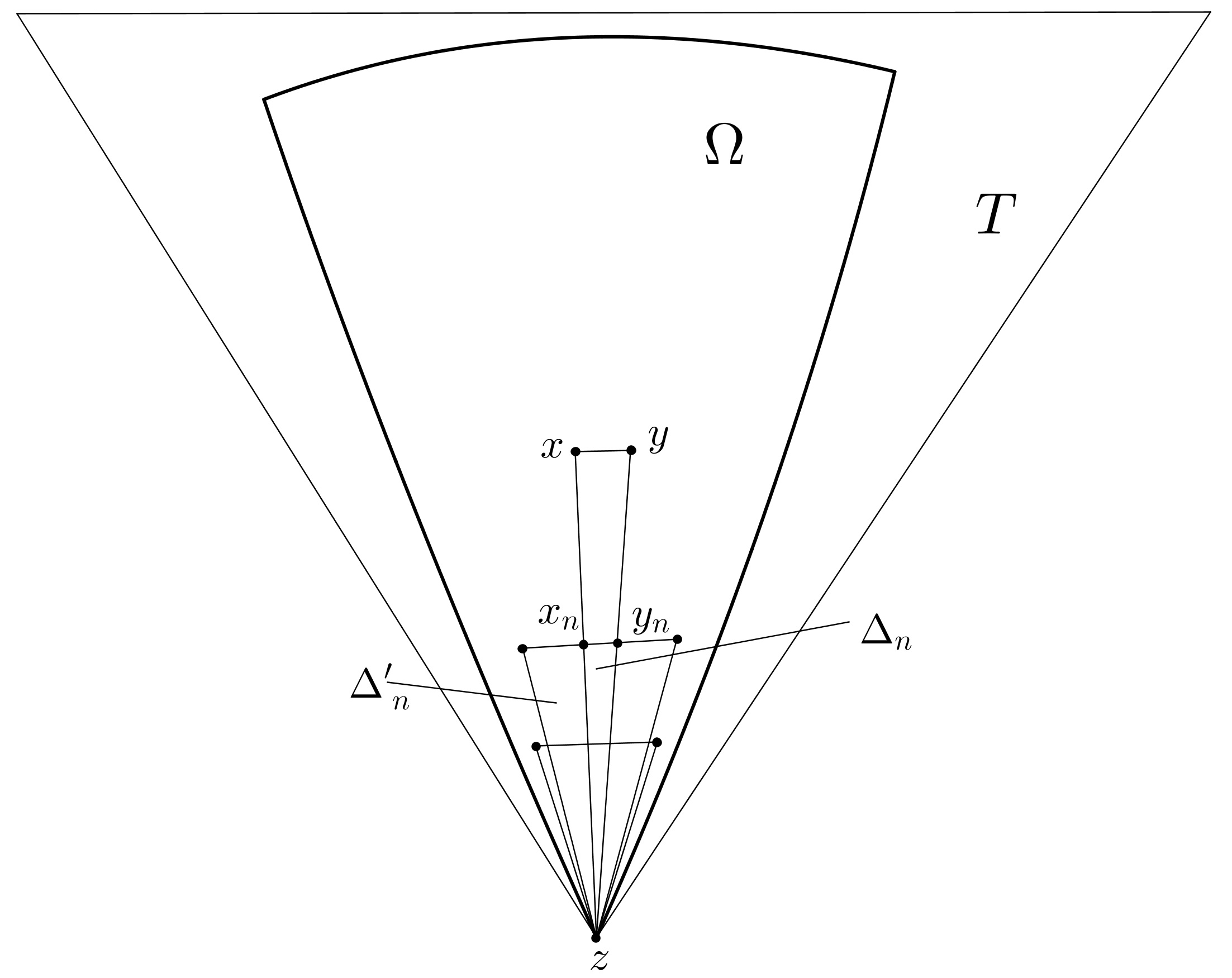}
\caption{$\C^1$ proof}
\label{gh_C1_fig}
\end{figure}
\end{center}

Now, choose two points $x$ and $y$ in $\O \cap \Pi$ and two sequences $x_n$ and $y_n$ of points of $[x,z[$ and $[y,z[$ converging to $z$. Consider the triangle $\Delta_n$ of vertices $x_n,y_n$ and $z$.

We want to show that $\underset{u \in [x_n,y_n]}{\sup} d_{\O}(u,[x_n,z] \cup [z,y_n]) \to \infty$. This needs some attention. One should remark that the comparison Theorem \ref{compa} shows that for any $u \in [x_n,y_n]$, we have $d_{T}(u,[x_n,z] \cup [z,y_n]) \leqslant d_{\O}(u,[x_n,z] \cup [z,y_n]) $. But the sequence $d_{T}(u,[x_n,z] \cup [z,y_n])$ is eventually constant and strictly positive. In particular, the sequence $d_{T}(u,[x_n,z] \cup [z,y_n])$ does go to $\infty$, unfortunately.

Nevertheless, one must remember that $\O$ must be strictly convex thanks to the previous paragraph. Hence, we can find a triangle $\Delta'_n$ whose sides tend to the side of $T$. Such a triangle contradicts Gromov-hyperbolicity.

\end{proof}

\subsection{The duality step}

We recall the following proposition.

\begin{proposition}\label{prop_dual_classi}
Let $\O$ be a properly convex open set. Then $\O$ is strictly convex if and only if $\dO^*$ is $\C^1$.
\end{proposition}

The following proposition is very simple once you know the notion of virtual cohomological dimension. The \emph{cohomological dimension} of a torsion-free group of finite type $\G$ is an integer $d_{\G}$ such that if $\G$ acts properly on $\R^d$ then $d\geqslant d_{\G}$ and the quotient $\Quotient{\R^d}{\G}$ is compact if and only if $d=d_{\G}$. The \emph{virtual cohomological dimension} of a virtually torsion-free\footnote{We recall that Selberg's lemma shows that every finite type subgroup of $\mathrm{GL}_m(\mathbb{C})$ is virtually torsion-free.} group is the cohomological dimension of any of its torsion-free finite-index subgroup. A reference is \cite{MR0422504}.

\begin{proposition}\label{prop_dual_action}
Let $\G$ be a discrete group of $\ss$ acting on a properly convex open set $\O$. Then the action of $\G$ on $\O$ is cocompact if and only if the action of $\G$ on $\O^*$ is cocompact.
\end{proposition}

We grab the opportunity to mention the following open question:

\begin{oqu}
Let $\G$ be a discrete group of $\ss$ acting on a properly convex open set $\O$. Is it true that the action of $\G$ on $\O$ is of finite covolume if and only if the action of $\G$ on $\O^*$ is of finite covolume ?
\end{oqu}

\par{
It is known that the answer is yes in dimension 2 from \cite{Marquis:2009kq} and also yes in any dimension if you assume that the convex set is strictly convex or with $\C^1$ boundary by \cite{Cooper:2011fk} and \cite{Crampon:2012fk}.
}

\subsection{The key lemma}

\begin{theorem}\label{thm_ghyp_dur}
Suppose that $\O$ is divisible and that $\O$ is strictly convex. Then $\O$ is Gromov-hyperbolic.
\end{theorem}

This theorem is the heart of the proof of Theorem \ref{thm_ghyp}. We postpone its proof until \ref{Ghyp_leretour} where we will show a more general theorem (Corollary \ref{coro_ghyp_leretour}).

\begin{proof}[Proof of Theorem \ref{thm_ghyp} assuming Theorem \ref{thm_ghyp_dur}]
\par{
$1) \Rightarrow 2)$ and $1) \Rightarrow 3)$ are the content of Proposition \ref{prop_ghyp_facile}.
}
\\
\par{
 $2) \Rightarrow 1)$ is the content of Theorem \ref{thm_ghyp_dur}.
}
\\
\par{ 
The equivalence $1) \Leftrightarrow 4)$ is a consequence of the fact that since $\G$ divides $\O$, $\G$ with the word metric is quasi-isometric to $(\O,d_{\O})$ and Gromov-hyperbolicity is invariant by quasi-isometry.
}
\\
\par{
Let us show that $3) \Rightarrow 4)$ to finish the proof. The properly convex open set $\O^*$ dual to $\O$ is strictly convex by Proposition \ref{prop_dual_classi} and the action of $\G$ on $\O^*$ is cocompact by Proposition \ref{prop_dual_action}. Therefore, by Theorem \ref{thm_ghyp_dur} the metric space $(\O^*,d_{\O^*})$ is Gromov-hyperbolic. But the group $\G$ acts by isometries and cocompactly on $(\O^*,d_{\O^*})$, hence the group $\G$ is Gromov-hyperbolic.
}
\end{proof}

Theorem \ref{thm_ghyp} has a fascinating corollary:

\begin{corollary}
Suppose a group $\G$ divides two properly convex open sets $\O$ and $\O'$. Then $\O$ is strictly convex if and only if $\O'$ is strictly convex.
\end{corollary}

We also want to stress that Theorem \ref{thm_ghyp} is the only way known by the author to show that a divisible convex set\footnote{which is not an ellipsoid.} of dimension at least 3 is strictly convex. Hence, the actual proof of the existence of a strictly convex divisible convex set relies on this theorem.

We shall see in subsection \ref{rigidity_round} that Gromov-hyperbolicity also implies some regularity for the boundary of $\O$.

\section{Moduli spaces}

\subsection{A naturality statement}

Let $\G$ be a group of finite type and $d \geqslant 2$ an integer. The set of homomorphisms $\Hom(\G,\ss)$ can be identified with a Zariski closed subspace of $\ss^N$, where $N$ is the number of generators of $\G$. We put on it the topology induced from $\Hom(\G,\ss)$. We denote by $\beta_{\G}$ the subspace of representations $\rho$ of $\G$ in $\ss$ which divide a non-empty properly convex open set $\O_{\rho}$ of $\PP^d$.

The following striking theorem shows that convex projective structures are very natural.

\begin{theorem}[Koszul - Benoist] \label{natur_stat}
Suppose $M$ is a compact manifold of dimension $d$ and $\G=\pi_1(M)$ does not contain an infinite nilpotent normal subgroup. Then the space $\beta_{\G}$ is a union of connected components of the space $\Hom(\G,\ss)$.
\end{theorem}

We will see in the next paragraph that Koszul showed the ``open'' part of this theorem in \cite{MR0239529}. The ``closed'' part has been shown by Choi and Goldman when $d=2$ in \cite{MR1145415}, by Kim when $d=3$ and $\G$ is a uniform lattice of $\so{3}$ in \cite{MR1913941} and finally Benoist showed the general case in \cite{MR2195260} (which is the hard step of the proof). One can also find a version of this theorem in the finite volume context in dimension 2 in \cite{MR2740643}.

\begin{remark}
In \cite{MR2195260} (Corollary 2.13), Benoist showed that if $\G$ divides a properly convex open set then the group $\G$ does not contain an infinite nilpotent normal subgroup if and only if $\G$ is strongly irreducible.
\end{remark}

\subsection{A sketch of proof of the closedness}

The proof of the fact that $\beta_{\G}$ is closed is quite hard in the general case. In dimension 2 we can give a simpler proof due to Choi and Goldman \cite{MR1145415}. The following proposition is a corollary of a classical theorem of Zassenhaus (\cite{64.0961.06}):

\begin{proposition}\label{limi_dis_faith}
Let $\G$ be a discrete group that does not contain an infinite nilpotent normal subgroup. Then any limit of a sequence of discrete and faithful representations $\rho_n: \G \to \ss$ is also discrete and faithful.
\end{proposition}

One can find a proof of this proposition in \cite{MR972343} or \cite{MR2648674}. This theorem explains the beginning of the story.

\begin{lemma}[Choi-Goldman \cite{MR1145415}]\label{irreducible}
Let $\G$ be a discrete group such that $\G$ does not admit any infinite nilpotent normal subgroup. Let $(\O_n,\rho_n)$ be a sequence where $\O_n$ is a properly convex open set of $\PP^d$ and $\rho_n$ a sequence of discrete and faithful representations, $\rho_n: \G \to \ss$, such that $\rho_n(\G)$ is a subgroup of $\Aut(\O_n)$. Suppose that $\rho_n \to \rho_{\infty} \in \Hom(\G,\ss)$. If $\rho_{\infty}$ is irreducible, then $\rho_{\infty}$ preserves a properly convex open set.
\end{lemma}

\begin{proof}
Endow the space of closed subset of $\PP^d$ with the Hausdorff topology. The subspace of closed convex subsets of $\PP^d$ is closed for this topology, therefore it is compact. Consider an accumulation point $K$ of the sequence $(\overline{\O_n})_{n \in \N}$ in this space. The convex set $K$ is preserved by $\rho_{\infty}$.

We have to show that $K$ has non-empty interior and is properly convex. To do that, consider the 2-fold cover $\S^d$ of $\PP^d$ and rather than taking an accumulation point $K$ of $(\overline{\O_n} )_{n \in \N}$ in $\PP^d$, take it in $\S^d$ and call it $Q$. We have three possibilities for $Q$:

\begin{enumerate}
\item $Q$ has empty interior;
\item $Q$ is not properly convex;
\item $Q$ is properly convex and has a non-empty interior.
\end{enumerate}

Suppose $Q$ has empty-interior. As $Q$ is a convex subset of $\S^d$, this implies that $Q$ spans a non-trivial subspace of $\R^{d+1}$ which is preserved by $\rho_{\infty}$. This is absurd since $\rho_{\infty}$ is irreducible. Suppose now that $Q$ is not properly convex. Then $Q \cap -Q \neq \varnothing$ spans a non-trivial subspace of $\R^{d+1}$. This is again absurd.

Therefore $Q$ is properly convex and has non-empty interior.
\end{proof}

In dimension 2, using lemma \ref{trace},  one can show easily that every accumulation point of $\beta_{\G}$ is an irreducible representation, hence the preceding lemma shows that $\beta_{\G}$ is closed. Unfortunately, this strategy does not work in high dimensions and the machinery to show closedness is highly more involved.

\begin{lemma}\label{trace}
Every Zariski-dense subgroup of $\s{2}$ contains an element with a negative trace. Every element $\g$ of $\s{3}$ preserving a properly convex open set is such that $\mathrm{Tr}(\g) \geqslant 3$.
\end{lemma}

\begin{remark}
The second assertion of this lemma is a trivial consequence of Proposition \ref{classi_matrix}. The first part is a lemma of \cite{MR1145415}.
\end{remark}

\begin{proof}[Proof of the closedness in dimension 2 assuming Lemma \ref{trace}]
Take a sequence $\rho_n \in \beta_{\G}$ which converges to a representation $\rho$. Suppose $\rho$ is not irreducible. Then up to conjugation and transposition, the image of $\rho$ restricted to $[\G,\G]$ is included in the subgroup of elements $\s{3}$ which fix every point of a line of $\R^3$ and preserve a supplementary plane (this subgroup is of course isomorphic to $\s{2}$). Hence, Lemma \ref{trace} shows that there exists an element $\g \in \G$ such that $\mathrm{Tr}(\rho_{\infty}(\g))< 1$. But Lemma \ref{trace} shows that for every element $\g$, we have $\mathrm{Tr}(\rho_n(\g)) \geqslant 3$. Hence, the representation $\rho$ is irreducible.

Therefore by Lemma \ref{irreducible}, the representation $\rho$ preserves a properly convex open set $\O$. The action of $\G$ on $\O$ is proper using the Hilbert metric and the quotient is compact since the cohomological dimension of $\G$ is 2.
\end{proof}

\subsection{The openness}

Let $M$ be a manifold. A \emph{projective manifold}\footnote{All projective structures are assume to be flat along this chapter. Here is an alternative definition: two torsion free connections on a manifold $M$ are \emph{projectively equivalent} if they have the same geodesics, up to parametrizations. A class of projectively equivalent connections defines a \emph{projective structure} on $M$. A projective structure is \emph{flat} if every point has a neighbourhood on which the projective structure is given by a flat torsion-free connection.}\index{projective manifold} is a manifold with a \, $(\PP^d,\mathrm{PGL}_{d+1}(\R))$-structure\footnote{See \cite{MR1435975} or \cite{MR957518} for the definition of a $(G,X)$-structure.}. A \emph{marked projective structure}\index{marked projective structure} on $M$ is a homeomorphism $h:M \to \mathcal{M}$ where $\mathcal{M}$ is a projective manifold. A marked projective structure is \emph{convex}\index{marked projective structure!convex} when the projective manifold $\mathcal{M}$ is convex, e.g. the quotient of a properly convex open set $\O$ by a discrete subgroup $\G$ of $\Aut(\O)$. Two marked projective structures $h,h': M \to \mathcal{M}, \mathcal{M}'$ are \emph{equivalent}\index{marked projective structure!equivalent} if there exists an isomorphism $i$ of $(\PP^d,\mathrm{PGL}_{d+1}(\R))$-structures between $\mathcal{M}$ and $\mathcal{M}'$ and the homeomorphism $h'^{-1}\circ i \circ h:M \to M$ is isotopic to the identity.

We denote by $\PP(M)$ the space of marked projective structures on $M$ and by $\beta(M)$ the subspace of convex projective structures. There is a natural topology on $\PP(M)$, see \cite{MR957518} for details.

\begin{theorem}[Kozsul, \cite{MR0239529}]\label{thm_kosz}
Let $M$ be a compact manifold. The subspace $\beta(M)$ is open in $\PP(M)$.
\end{theorem}

One can find a proof of this theorem in \cite{MR0239529}, but also in a lecture of Benoist \cite{CoursDeBenoist} or in a paper of Labourie \cite{MR2402597}. In the last reference, Labourie did not state the theorem but his Theorem 3.2.1 implies Theorem \ref{thm_kosz}.

The formalism of Labourie is the following: a projective structure on a manifold $M$ is equivalent to the data of a torsion-free connexion $\nabla$ and a symmetric 2-tensor $h$ satisfying a compatibility condition. Namely, let $\nabla^h$ be a connexion on the vector bundle $TM\times \R \to M$ given by the formula: $$\nabla_X^h(Z,\lambda)=\big(\nabla_X Z+\lambda X, L_X(\lambda)+h(Z,X)\big)$$ where $X,Z$ are vector fields on $M$, $\lambda$ is a real-valued function on $M$ and $L_X \lambda$ is the derivation of $\lambda$ along $X$. We say that the pair $(\nabla,h)$ is \emph{good} if $\nabla$ preserves a volume and $\nabla^h$ is flat.

Labourie shows that every good pair defines a projective structure and that every projective structure defines a good pair. Finally, Labourie shows that a projective structure is convex if and only if the symmetric 2-tensor $h$ of the good pair is definite positive. Hence, he shows that being convex is an open condition.

\subsection{Description of the topology for the surface}

Let $\Sigma$ be a compact surface with a finite number of punctures. We denote by $\beta(\Sigma)$ and $\mathrm{Hyp}(\Sigma)$ the moduli spaces of marked convex projective structures and hyperbolic structures on $\Sigma$,  and by $\beta_f(\Sigma)$ and $\mathrm{Hyp}_f(\Sigma)$ the finite-volume ones (for the Busemann volume).

In dimension 2, we can give three descriptions of the moduli space $\beta(\Sigma)$. The first one comes from Fenchel-Nielsen coordinates, the second one comes from Fock-Goncharov coordinates on Higher Teichm\"uller space and finally the third one was initiated by Labourie and Loftin independently.

\begin{theorem}[Goldman]
Suppose that $\Sigma$ is a compact surface with negative Euler characteristic $\chi$. Then the space $\beta(\Sigma)$ is a ball of dimension $-8\chi$.
\end{theorem}

Goldman shows this theorem for compact surfaces in \cite{MR1053346}\footnote{which is by the way a very nice door to the world of convex projective geometry.} by extending Fenchel-Nielsen coordinates on Teichm\"uller space. Choi and Goldman extend this theorem to compact 2-orbifold in \cite{MR2170138}. Finally, the author of this chapter extends Goldman's theorem to the case of finite-volume surfaces in \cite{MR2740643}.

In a different time by completely different methods, Fock-Goncharov find a system of coordinates for higher Teichm\"uller space described in \cite{MR2233852}. Since the space $\beta(\Sigma)$ is the second simplest higher Teichm\"uller space (after Teichm\"uller space itself), their results give another system of coordinates on $\beta(\Sigma)$. The situation of $\s{3}$ is much simpler than the situation of a general semi-simple Lie group, and there is a specific article dealing with Fock-Goncharov coordinates in the context of $\beta(\Sigma)$, that the reader will be happy to read: \cite{MR2304317}. Note that the Fock-Goncharov coordinates are nice enough to describe very simply and efficiently $\beta_f(\Sigma)$ in $\beta(\Sigma)$.

This leads us to the last system of coordinates:

\begin{theorem}[Labourie \cite{MR2402597} or Loftin \cite{MR1828223} (compact case), Benoist-Hulin (finite volume case \cite{ben_hulin})]\label{complex_moduli}
There exists a fibration $\beta_f(\Sigma) \to \mathrm{Hyp}_f(\Sigma)$ which is equivariant with respect to the mapping class group of $\Sigma$.
\end{theorem}

\subsection{Description of the topology for 3-orbifolds}

\subsubsection{Description of the topology for the Coxeter Polyhedron.}

There is an another context where we can describe the topology of $\beta_{\G}$. It is the case of certain Coxeter groups in dimension 3.

Take a polyhedron $\GG$, label each of its edges $e$ by a number $\theta_e \in ]0,\frac{\pi}{2}]$ and consider the space $\beta_{\GG}$ of marked\footnote{This means that you keep track of a numbering of the faces of $P$.} polyhedra $P$ of $\PP^3$ with a reflection $\sigma_s$ across each face $s$ such that for every two faces $s,t$ of $P$ sharing an edge $e= s \cap t$, the product of two reflections $\sigma_s$ and $\sigma_t$ is conjugate to a rotation of angle $\theta_e$.

We define the quantity $d(\GG) = e^+-3$ where $e^+$ is the number of edges of $\GG$ not labelled $\frac{\pi}{2}$.

We need an assumption on the shape of $P$ to get a theorem. Roughly speaking, the assumption means that $P$ can be obtained by induction from a tetrahedron by a very simple process.

This process is called ``\'ecimer'' in french. The english translation of this word is ``to pollard'' i.e cutting the head of a tree. An \emph{ecimahedron } is a polyhedron obtained from the tetrahedron by the following process: see Figure \ref{ecim}. 

\begin{center}
\begin{figure}[h!]
\centering
\includegraphics[width=7cm]{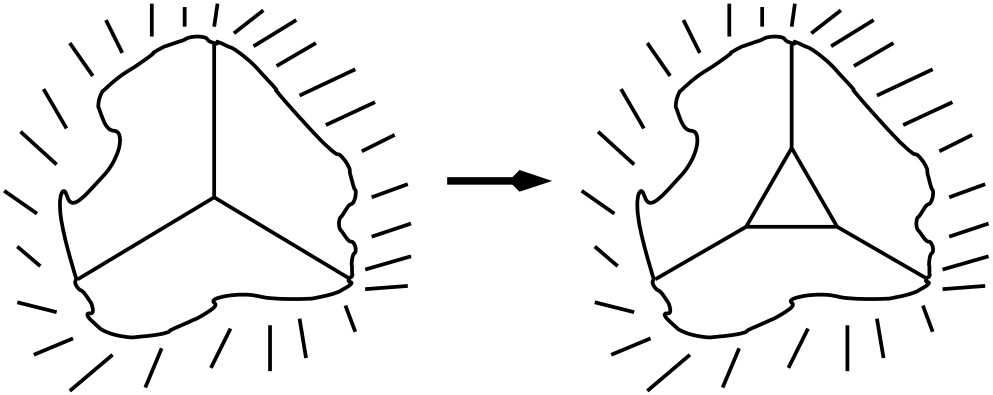}
\caption{Ecimation or truncation\label{ecim}}
\end{figure}
\end{center}

A \emph{bad 3-circuit} of $\GG$ is a sequence of three faces $r,s,t$ that intersect each other but such that the edges $e = r \cap s$, $f= s \cap t$ and $g = t \cap s$ do not intersect and such that $\theta_e +\theta_f+\theta_g \geqslant \pi$ and one of the angles $\theta_e , \theta_f,\theta_g$ is equal to $\frac{\pi}{2}$.

\begin{theorem}[Marquis. \cite{MR2660566}]\label{ecima}
Suppose $\GG$ is an ecimahedron without bad 3-circuits, which is not a prism and such that $d(\GG) > 0$. Then the space $\beta_{\GG}$ is a finite union of balls of dimension $d(\GG)$. Moreover, if $\GG$ is a hyperbolic polyhedron, then $\beta_{\GG}$ is connected.
\end{theorem}

The number of connected components of $\beta_{\GG}$ can be computed but we refer the reader to \cite{MR2660566} for a statement.

The expression ``$\GG$ is an hyperbolic polyhedron'' means that there exists a hyperbolic polyhedron which realizes $\GG$. Andreev's theorem perfectly answers this question (see \cite{MR0259734} or \cite{MR2336832}).

The assumption that $\GG$ is not a prism is here just to simplify the statement. The case of the prism can be worked out very easily.

We also mention that Choi shows in \cite{MR2247648} that under an ``ordonnability'' condition, $\beta(\GG)$ is a manifold of dimension $d(\GG)$. Choi, Hodgson and Lee in \cite{Choi:2010ys} and Choi and Lee in \cite{Choi:2012kx} study the regularity of the points corresponding to the hyperbolic structure in $\beta(\GG)$.

\subsubsection{Description of the topology for 3-manifolds}

Heusener and Porti show the ``orthogonal'' theorem to Theorem \ref{ecima}. Namely, in practice, Theorem \ref{ecima} shows that one can find an infinite number of hyperbolic polyhedra such that their moduli space of convex projective structures is of dimension as big as you wish. Heusener and Porti show:

\begin{theorem}[Heusener and Porti,\cite{MR2860986}]\label{heusener}
There exist an infinite number of compact hyperbolic 3-manifolds such that $\beta(M)$ is a singleton.
\end{theorem}
\par{
Cooper, Long and Thistlethwaite compute explicitly, using a numerical program, the dimension of $\beta(M)$ for $M$ a compact hyperbolic manifold, using an explicit presentation and a list of 3-manifolds in \cite{MR2264468}.
}
\\
\par{
It would be very interesting to find a topological obstruction to the existence of a deformation. And also a topological criterion for the smoothness of the hyperbolic point in $\beta(M)$, as started by Choi, Hodgson and Lee. We also want to mention that very recently Ballas obtained  in \cite{Ballas:2012fk} a version of Theorem \ref{heusener} in the context of finite-volume convex manifolds.
}
\section{Rigidity}

For us a rigidity result is a theorem that says: ``If a convex set have some regularity plus a big isometry group then this convex has to belong to this list''. We give three precise statements. The interested reader can also read the chapter \cite{Guo} of this handbook.

\subsection{For strongly convex bodies}

We say that a properly convex open set is \emph{strongly convex}\index{strongly convex}\index{convex!strongly convex} when its boundary is $\C^2$ with positive Hessian. Colbois and Verovic show  in \cite{MR2057242} that a strongly convex body of $\PP^d$ endowed with its Hilbert metric is bi-lipschitz equivalent to the hyperbolic space of dimension $d$.

\begin{theorem}[Sasaki \cite{MR1165117}, Soci\'e-M\'ethou \cite{MR1981171}]
Let $\O$ be a strongly convex open set. If the group $\Aut(\O)$ is not compact, then $\O$ is an ellipsoid.
\end{theorem}

\par{
Sasaki proves this theorem in \cite{MR1165117} in the case where the boundary is $\C^{\infty}$ and $d \geqslant 4$. Podest\`a gives  in \cite{MR1164539} a detailed proof of Sasaki's theorem with some refinements. Finally Soci\'e-Methou gives  in \cite{MR1981171} the proof in full generality. In fact, Soci\'e-Methou shows a more precise statement; she shows that \emph{if a properly convex open set admits an infinite order automorphism which fixes a point $p\in \dO$, then if $\dO$ admits an osculatory ellipsoid at $p$ then $\O$ is an ellipsoid}.
}
\\
\par{
We remark that the technique of Sasaki-Podest\`a and Soci\'e-Methou are completely different. The first one uses geometry of affine spheres and the second one uses only elementary techniques.
}

\subsection{For round convex bodies}\label{rigidity_round}

The following theorem shows that the boundary of a quasi-divisible convex open set cannot be too regular unless it is an ellipsoid.

\begin{theorem}[Benoist (Divisible) \cite{MR2094116}, Crampon-Marquis (Quasi-divisible) \cite{Mickael:2012fk}]
Let $\O$ be a quasi-divisible strictly convex open set. Then the following are equivalent:
\begin{enumerate}
\item The convex set $\O$ is an ellipsoid.
\item The regularity of the boundary of $\O$ is $\C^{1+\varepsilon}$ for every $0 \leqslant \varepsilon < 1$.
\item The boundary of $\O$ is $\beta$-convex for every $\beta > 2$.
\end{enumerate}
\end{theorem}
\par{
Let $\alpha >0$ and consider the map $\phi:\R^d \rightarrow \R$, $x \mapsto |x|^{\alpha}$, where $|\cdot|$ is the canonical Euclidean norm. The image of $\phi$ defines a properly convex open subset $\E_{\alpha}$ of $\PP^d$ which is analytic outside the origin and infinity. A point $p \in \dO$ is of \emph{class $\C^{1+\varepsilon}$}\index{convex!class $\C^{1+\varepsilon}$} (resp. \emph{$\beta$-convex}\index{convex!$\beta$-convex}) if and only if one can find an image of $\E_{\varepsilon}$ (resp. $\E_{\beta}$) by a projective transformation inside (resp. outside) $\O$ and such that the point origin of $\E_{\varepsilon}$ (resp. $\E_{\beta}$) is sent to $p$.
}
\\
\par{
We recall that a quasi-divisible strictly convex open set $\O$ is Gromov-hyperbolic, therefore there exists an $\varepsilon > 0$ and a $\beta > 2$ such that the boundary of $\O$ is $\C^{1+\varepsilon}$ and $\beta$-convex (\cite{MR2010741}). The reader should find more details about this theorem in the chapter ''The geodesic flow on Finsler and Hilbert geometries'' in this handbook.
}

\subsection{For any convex bodies}

The following statement is just a reformulation of Theorem \ref{thm_zari1}. We stress that this statement can be seen as a rigidity theorem.

\begin{theorem}[Benoist \cite{MR2010735}]
Let $\O$ be a divisible convex open set. Then the following are equivalent:
\begin{enumerate}
\item The group $\Aut(\O)$ is not Zariski-dense.
\item The convex set $\O$ is symmetric.
\end{enumerate}
\end{theorem}

\section{Benz\'ecri's theorem}

Let $X^{\bullet}=\{(\O,x) \,|\, \O \textrm{ is a properly convex open set of } \PP^d \textrm{ and } x \in \O\}$. The group $\PG$ acts naturally on $X^{\bullet}$. The following theorem is fundamental in the study of Hilbert geometry.

\begin{theorem}[Benz\'ecri's Theorem]\label{benzecri}
The action of $\PG$ on $X^{\bullet}$ is proper and cocompact.
\end{theorem}

Roughly speaking, even if the action of $\PG$ on $X^{\bullet}$ is not homogeneous the quotient space is compact; therefore it means that we can find some homogeneity in the local geometric properties of Hilbert geometry. Precise examples of this rough statement can be found in Subsection \ref{appli_ben} of the present text.

\subsection{The proof}

To prove Theorem \ref{benzecri}, we will need two lemmas. The first one is very classical, the second one is less classical.

\subsubsection{John's ellispoid}\label{sub_john}\index{John's ellipsoid}

\begin{lemma}\label{john}
Given a bounded convex subset $\O$ of the affine space $\R^d$ with center of mass at the origin there exists a unique ellipsoid $\E$ with center of mass at the origin which is included in $\O$ and of maximal volume. Moreover, we have $\E \subset \O \subset d\E$.
\end{lemma}

\subsubsection{Duality}\label{duality2}
$\,$\\
\par{
The following lemma can sound strange to the reader not used to jump from the projective world to the affine world and vice versa; nevertheless once we understand it, the lemma should sound right; but its proof needs some analytic tools.
}

\begin{lemma}\label{centerofmass}
Given a properly convex open set $\O$ of $\PP^d$ and a point $x \in \O$, there exists a unique hyperplane $x^\star \in \O^*$ such that $x$ is the center of mass of $\O$ viewed in the affine chart $\PP^d \smallsetminus x^{\star}$. Moreover, the map $X^{\bullet}\rightarrow (\PP^d)^*$ defined by $(\O,x) \mapsto x^\star$ is continuous, $\PG$-equivariant and the map $\O \rightarrow \O^*$ given by $x \mapsto x^\star$ is analytic.
\end{lemma}

We first give a
\begin{proof}[Proof of Benz\'ecri's Theorem assuming Lemmas \ref{john} and \ref{centerofmass}]
\par{
Consider the sp\-ace $S^{\bullet} =\{(\E,x) \,|\, \E \textrm{ is an ellipsoid of } \PP^d \textrm{ and } x \in \E\}$. The action of $\PG$ on $S^{\bullet}$ is transitive with stabilizer $\mathrm{PSO}_{d}(\R)$, a compact subgroup of $\PG$, so $S^{\bullet}$ is a $\PG$-homogeneous space on which $\PG$ acts properly.
}
\\
\par{
We are going to define a fiber bundle of $X^{\bullet}$ over $S^{\bullet}$ with compact fiber which is  $\PG$-equivariant. Namely, the map $\varphi:X^{\bullet}\rightarrow S^{\bullet}$ which associates to the pair $(\O,x)$ the pair $(\E,x)$ where $\E$ is the John ellipsoid of $\O$ viewed in the affine chart $\PP^d \smallsetminus x^\star$. The map $\varphi$ is continuous and well-defined, thanks to Lemma \ref{centerofmass}.
}
\\
\par{
Lemma \ref{john} shows that the fibers of this map are compact and this map is clearly equivariant. Since the action of $\PG$ on $S^{\bullet}$ is proper and cocompact, it follows that the action of $\PG$ on $X^{\bullet}$ is proper and cocompact.
}
\end{proof}

\subsubsection{A sketch of the proof of Lemma \ref{centerofmass}.}

\par{
In \cite{MR0158414} Vinberg introduces a natural diffeomorphism between a sharp convex cone and its dual. Let $\C$ be a sharp convex cone of $\R^{d+1}$. The map
}

$$
\begin{array}{cccc}
\varphi_*: & \C^* & \rightarrow & \C \\
                      & \psi & \mapsto & \textrm{Center of mass}(\C_\psi (d+1))
\end{array}
$$
$\textrm{ where } \C_{\psi}(d+1) = \{ u \in \C \,|\, \psi(u)  = d+1\}$, is an \textit{equivariant diffeomorphism}.

For a proof of this statement, the reader can consult \cite{MR0158414} but also \cite{NoteGoldman}. We give a geometric presentation of this diffeomorphism but in practice an analytic presentation is needed to understand it correctly.

\begin{proof}[Proof of Lemma \ref{centerofmass}]
\par{
Let $\O$ be a properly convex open set and $x$ a point in $\O$. Consider $\C$ and $\C^*$, the cone above $\O$ and $\O^*$ respectively, and the diffeomorphism $\varphi_* : \C^* \rightarrow \C$.
}
\\
\par{
Take a point $u$ in $\C$ such that $[u]=x$. Let $\psi \in \C^*$ be the linear form defined by $\psi = \varphi_*^{-1}(u)$. Then $\psi(u)=d+1$ and $u$ is the center of mass of $\C_\psi (d+1)$. Therefore, $x$ is the center of mass of $\O$ in the affine chart $\PP^d \smallsetminus \psi^{-1}(0)$. In other words, $x^\star=[\psi] \in \O^*$ is the point we are looking for.
}
\end{proof}

\subsection{Natural things are equivalent}\label{appli_ben}

\subsubsection{Definitions}$\,$\\

We denote by  $X$ the following space:
$$X=\{\O \,|\, \O \textrm{ is a properly convex open subset of } \PP^d \}.$$

For us, a \emph{projective volume}\index{projective volume}\index{projective!volume}\index{volume!projective} is a map $\mu$ from $X$ which associates to a properly convex open set $\O$ an absolutely continuous measure $\mu_{\O}$ on $\O$ with respect to Lebesgue measure. A \emph{projective metric}\index{projective metric}\index{projective!metric}\index{metric!projective} is a map $F$ from $X^{\bullet}$ which associates to a pair $(\O,x) \in X^{\bullet}$ a norm $F_{\O}(x)$ on the tangent space $T_x \O$ of $\O$ at $x$. Each of these notions is said to be \emph{natural}\index{natural}\index{projective!natural}\index{projective metric!natural}\index{projective volume!natural} when it is invariant by $\PG$ and continuous.

The two basic examples are the Hilbert distance which gives rise to the Busemann volume and the Holmes-Thompson volume. We give a brief definition of both volume in the introduction. We will denote by Hil any measure, distance, norm, etc... coming from the Hilbert distance.

\subsubsection{Volume}

\begin{proposition}
In every dimension $d$, given any natural projective volume $\mu$, there exist two constants $0<a_d<b_d$ such that for any properly convex open set $\O$ we have $a_d \mu_{\O} \leqslant \mu^{\textrm{Hil}}_{\O} \leqslant b_d \mu_{\O}$.
\end{proposition}

\begin{proof}
Given a properly convex open set $\O$, let us denote by $f_{\O}$ (resp. $g_{\O}$) the density of $\mu_{\O}$ (resp. $\mu^{\textrm{Hil}}_{\O}$) with respect to Lebesgue measure. We get a map from $X^{\bullet}$ to $\R^*_+$ given by $(\O,x) \mapsto \frac{f_{\O}(x)}{g_{\O}(x)}$. This map is continuous since $\mu_{\O}$ and $\mu^{\mathrm{Hil}}_{\O}$ are absolutely continuous. This map is also $\PG$ invariant. So, Theorem \ref{benzecri} shows that this map attains its maximum and its minimum which are two strictly positive constants.
\end{proof}

The two following corollaries are now trivial. We just recall some definitions for the convenience of the reader. Given a projective volume $\mu$, a properly convex open set $\O$ is said to be \emph{$\mu$-quasi-divisible} if there exists a discrete subgroup $\G$ of $\Aut(\O)$ such that $\mu\Big(\Quo\Big) < \infty$.

\begin{corollary}
If $\mu$ and $\mu'$ are two natural projective volumes then any properly convex open set $\O$ is $\mu$-quasi-divisible if and only if it is $\mu'$-quasi-divisible.
\end{corollary}

\par{
This corollary justifies the usual definition of quasi-divisible convex set which uses Busemann volume.
}
\\
\par{
Given a projective volume $\mu$, the  \emph{$\mu$-sup-volume entropy} of $\O$ is the quantity $$\underset{R \to \infty}{\overline{\lim}} \frac{\log(\mu(B(x,R)))}{R},$$ where $x$ is any point\footnote{This quantity does not depend on $x$.} of $\O$ and $B(x,R)$ is the ball of radius $R$ of $(\O,d^{\mathrm{Hil}}_{\O})$.
}

\begin{corollary}
If $\mu$ and $\mu'$ are two natural projective volumes then for any properly convex open set $\O$ the $\mu-$sup-volume entropy and the $\mu'-$sup-volume entropy coincide.
\end{corollary}

Of course, the same result is true if we change the supremum into an infimum.

\subsubsection{Metric}

\begin{proposition}
In every dimension $d$, given any natural projective metric $F$ there exist two constants $0<a_d<b_d$ such that for any properly convex open set $\O$ and for any point in $x$, we have $a_d F_{\O}(x) \leqslant F^{Hil}_{\O}(x) \leqslant b_d F_{\O}(x)$.
\end{proposition}

\begin{proof}
We denote by $F^{Hil}$ the Hilbert metric and we introduce a slightly bigger space than $X^{\bullet}$. We take
$$X^\divideontimes= \{(\O,x,v) \,|\, (\O,x) \in X^{\bullet},\, v\in T_x\O \textrm{ and } F^{Hil}_{\O}(x)(v)=1\}.$$
The action of $\PG$ on $X^\divideontimes$ is again proper and cocompact since $X^\divideontimes$ is a $\PG$-equivariant fiber bundle over $X^{\bullet}$ with compact fiber.

Now the following map is continuous, $\PG$ invariant and takes strictly positive value: $(\O,x,v) \in X^\divideontimes \mapsto  F_{\O}(x)(v)$. The conclusion is straightforward.
\end{proof}

\begin{corollary}
If $\d$ and $\d'$ are two distances coming from natural projective metrics then for any properly convex open set $\O$, the metric spaces $(\O,d_{\O})$ and $(\O,d'_{\O})$ are bi-Lipschitz equivalent through the identity map.
\end{corollary}

We remark that there is no reason for which the volume entropy of two bi-Lipschitz spaces is the same. We also remark that if a convex set is Gromov-hyperbolic for the Hilbert distance, it is Gromov-hyperbolic for all natural distances.

\subsubsection{Curvature.}

It is hard to give a meaning to the sentence ``Hilbert geometries are non-positively curved'' since the Hilbert distance is a Finsler metric and not a Riemannian metric. A Hilbert geometry which is CAT(0) is an ellipsoid, see \cite{MR1456516}. 

Lemma \ref{metric_riema} allows to construct two natural \emph{Riemmanian} metrics on a Hilbert geometry, $g_{vin}$ and $g_{aff}$. The first one is construct thanks to Vinberg hypersurface (Theorem \ref{surf_vinberg}) and the second one thanks to affine hypersurface (Theorem \ref{surf_yau}).

\begin{proposition}
In any dimension, there exist two numbers $\kappa_1 \leqslant \kappa_2$ such that for any properly convex open set, any point $x \in \O$ and any plane $\Pi$ containing $x$, the sectional curvature of $g_{vin}$ and $g_{aff}$ at $(x,\Pi)$ is between $\kappa_1$ and $\kappa_2$.
\end{proposition}

The triangle gives an example where the curvature is constant and equal to $0$. The ellipsoid is an example where the curvature is constant equal to $-1$.

It would have been good if $\kappa_2 \leqslant 0$. Unfortunately, Tsuji shows in \cite{MR688133} that there exists a properly convex open set $\O$, a point $x$ and a plane $\Pi$ such that the sectional curvature of $g_{aff}$ at $(x,\Pi)$ is strictly positive. Since the example of Tsuji is a homogeneous properly convex open set, we stress that $g_{vin} = g_{aff}$ for this example.

Hence, Hilbert geometry cannot be put in the world of non-positively curved manifolds using a Vinberg hypersurface or an affine sphere.

Finally, we note that Calabi shows in \cite{MR0365607} that the Ricci curvature of $g_{aff}$ is always non-positive.

\subsection{Two not-two-lines applications of Benz\'ecri's theorem}

\subsubsection{The Zassenhaus-Kazhdan-Margulis lemma in Hilbert Geometry}

The use of the Margulis constant has proved to be very useful in the study of manifolds of non-positive curvature. The following lemma says that this tool is also available in Hilbert geometry.

\begin{theorem}[Choi \cite{MR1405450}, Crampon-Marquis \cite{Crampon:2011fk}, Cooper-Long-Tillmann \cite{Cooper:2011fk}]
In any dimension $d$, there exists a constant $\varepsilon_d$ such that for every properly convex open set $\O$, for every point $x \in \O$, for every real number $0 < \varepsilon < \varepsilon_d$ and for every discrete subgroup $\G$ of $\Aut(\O)$, the subgroup $\G_{\varepsilon}(x)$ generated by the set $\{ \g \in \G \,|\, d_{\O}(x,\g x) \leqslant \varepsilon \}$ is virtually nilpotent.
\end{theorem}

The following lemma of Zassenhaus is the starting point:

\begin{lemma}[Zassenhaus \cite{64.0961.06}]
Given a Lie group $G$, there exists a neighbourhood $\,\U$ of $e$ such that for any discrete group $\G$ of $G$, the subgroup $\G_{\U}$ generated by the set $\{ \g \in \G \,|\, \g \in \U \}$ is nilpotent.
\end{lemma}

First let us ``virtually'' prove the lemma in the case where $\O$ is an ellipsoid $\E$. Since the action of $\Aut(\E)$ on $\E$ is transitive and since the Margulis constant is a number depending only on the geometry of the space at the level of points, we just have to prove it for one point. We choose a point $O \in \E$ and we have to show that if $\varepsilon$ is small enough then every discrete group of $\Aut(\E)$ generated by an element moving $O$ a distance less that $\varepsilon$ is virtually nilpotent.

The Zassenhaus lemma gives us a neighbourhood $\U$ of $e$ in $\Aut(\E)$ such that for any discrete subgroup $\G$ of $G$, the subgroup $\G_{\U}$ generated by $\{ \g \in \G \,|\, \g \in \U \}$ is nilpotent. The open set $\mathcal{O}=\{ \g \in \Aut(\E) \,|\, d_{\E}(O,\g(O)) < \varepsilon \}$ is contained in the open set $\Stab_O \cdot \U$ if $\varepsilon$ is small enough.

What remain is that $\G_{\U}$ is of finite index $N$ in the group $\G_{\varepsilon}$ generated by $\G \cap \mathcal{O}$. For this we have to show that $N$ is less than the number of translates of $\U$ by $\Stab_O$ needed to cover entirely $\Stab_O$.

In the case of Hilbert geometry, we have to replace the fact that the action of $\Aut(\E)$ on $\E$ is transitive by Benz\'ecri's theorem \ref{benzecri}. The proof of Benz\'ecri's theorem gives us an explicit and simple compact space $D$ of $X^{\bullet}$ such that $\PG \cdot D = X^{\bullet}$. Namely, choose any affine chart $\A$, any point $O\in \A$, any scalar product on $\A$, denote by $B$ the unit ball of $\A$ and set:
$$
D=\{ (\O,O) \,|\, O \textrm{ is the center of mass of } \O \textrm{ in } \A \textrm{ and } B \subset \O \subset dB\},
$$
\noindent where $dB$ is the image of $B$ by the homothety of ratio $d$ and center at $O$ in $\A$.
We need to link the Hilbert distance with the topology on $\PG$. This is done by the following two lemmas:

\begin{lemma}\label{lem_mar1}
For every $\varepsilon > 0$, there exists a $\delta > 0$ such that for every  $(\O,O) \in D$ and every element $\g \in \Aut(\O)$, we have:
$$d_{\O}(O, \g(O)) \leqslant \varepsilon \Longrightarrow d_{\PG}(1,\g) \leqslant \delta.$$
\end{lemma}

\begin{lemma}\label{lem_mar2}
For every $\varepsilon > 0$, there exists a $\delta > 0$ such that for every $(\O,O) \in D$ and every element $\g \in \Aut(\O)$, we have:
$$d_{\O}(O, \g \cdot O) \leqslant \delta \Longrightarrow d_{\PG}(\Stab_{\O}(O),\g) \leqslant \varepsilon.$$
\end{lemma}

The rest of the proof is like in hyperbolic geometry. The details can be found in \cite{Cooper:2011fk} or \cite{Crampon:2011fk}; the strategy is the same and can be traced back to the two-dimensional proof of \cite{MR1405450}.

\subsubsection{A characterisation of Gromov-hyperbolicity using the closure of orbits under $\PG$}\label{Ghyp_leretour}$\,$\\

Recall that $X$ is the space of properly convex open subsets of $\PP^d$ endowed with the Hausdorff topology. For every $\delta > 0$, we denote by $X_{\delta}$ the space of properly convex open subsets which are $\delta$-Gromov-hyperbolic for the Hilbert distance.

\begin{theorem}[Benoist \cite{MR2010741}]$\,$\label{super_ghyp}
\begin{enumerate}
\item The space $X_{\delta}$ is closed.
\item Conversely, for every closed subset $F$ of $X$ which is $\PG$-invariant and contains only strictly convex open sets, there exists a constant $\delta$ such that $F \subset X_{\delta}$.
\end{enumerate}
\end{theorem}

\begin{proof}
\par{
We only sketch the proof. We start by showing the first point. Take a sequence $\O_n \in X_{\delta}$ converging to a properly convex open set $\O_{\infty}$. We have to show that $\O_{\infty}$ is $\delta$-hyperbolic; we first show that it is strictly convex.
}
\\
\par{
Suppose $\O_{\infty}$ is not strictly convex. Then there exists a maximal non-trivial segment $[x_{\infty},y_{\infty}] \subset \O_{\infty}$. Therefore, we can find (see the proof of Lemma \ref{prop_ghyp_facile}) a sequence of triangles $z_n,x_n,y_n$ included in $\O_{\infty}$ whose ``size'' in $\O_{\infty}$ tends to infinity. If $n$ is large enough these triangles are in fact included in $\O_n$ and so their size in $\O_n$ should be less than $\delta$. Passing to the limit we get a contradiction.
}
\\
\par{
To show that $\O_{\infty}$ is Gromov-hyperbolic, we do almost the same thing. Suppose it is not in $X_{\delta}$. Take a triangle of $\O_{\infty}$ whose size is more than $\delta$. Since $\O_{\infty}$ is strictly convex, the side of such a triangle is a segment, and one can conclude like in the previous paragraph.
}
\\
\par{
Now for the second affirmation. Suppose such a $\delta$ does not exist. Then we can find a sequence of triangles $T_n$ whose size goes to infinity in a sequence of properly convex open sets $\O_n$. Precisely, this means that there exists a sequence $x_n,y_n,z_n,u_n$ of points of $\O_n$ such that $u_n \in [x_n,y_n]$ and the quantity $d_{\O_n}(u_n,[x_n,z_n]\cup [z_n,y_n])$ tends to infinity. Using Benz\'ecri's theorem, we can assume that the sequence $(\O_n,u_n)$ converges to $(\O_{\infty},u_{\infty})$. Since $F$ is closed, we get that $\O_{\infty}$ is strictly convex. One can also assume that the triangle $x_n,y_n,z_n$ converges to a triangle of $\overline{\O_{\infty}}$, which cannot be degenerate since its size is infinite. But the strict convexity of $\O_{\infty}$ contradicts the fact that the size of the limit triangle is infinite.
}
\end{proof}

\begin{corollary}\label{coro_ghyp_leretour}
Let $\O$ be a properly convex open set of $\PP^d$. The following properties are equivalent:
\begin{enumerate}
\item The metric space $(\O,d_{\O})$ is Gromov-hyperbolic.
\item The closure of the orbit $\PG \cdot \O$ in $X$ contains only strictly convex properly convex open sets.
\item The closure of the orbit $\PG \cdot \O$ in $X$ contains only properly convex open sets with $\C^1$ boundary.
\end{enumerate}
\end{corollary}

\begin{proof}
We just do $1) \Leftrightarrow 2)$.
\par{
If the closure $F$ of the orbit $\PG \cdot \O$ in $X$ contains only strictly convex properly convex open sets, since $F$ is closed and $\PG$-invariant, Theorem \ref{super_ghyp} shows that all the elements of $F$ are in fact Gromov-hyperbolic, hence $\O$ is Gromov-hyperbolic.
}
\\
\par{
Now, if $\O$ is Gromov-hyperbolic, there is a $\delta > 0$ such that $\O \in X_{\delta}$. This space is a closed and $\PG$-invariant subset of $X$, hence the closure of $\PG \cdot \O$ is included in $X_{\delta}$. Hence Proposition \ref{prop_ghyp_facile} follows.
}
\end{proof}

This corollary has a number of nice corollaries. The first one is an easy consequence, using Proposition \ref{prop_dual_classi}.

\begin{corollary}
The metric space  $(\O,d_{\O})$ is Gromov-hyperbolic if and only if  $(\O^*,d_{\O^*})$ is Gromov-hyperbolic.
\end{corollary}

The second one is the conclusion of Theorem \ref{thm_ghyp} that we have left so far.

\begin{corollary}
A divisible convex set which is strictly convex is Gromov-hyperbolic.
\end{corollary}

To show this corollary, one just needs to know the following proposition which is a direct consequence of Benz\'ecri's theorem.

\begin{proposition}
Let $\O$ be a divisible convex set. Then the orbit $\PG \cdot \O$ is closed in $X$.
\end{proposition}

\begin{proof}
Benz\'ecri's theorem shows that the action of $\PG$ on $X^{\bullet}$ is proper, hence the orbit $\PG \cdot (\O,x)$ is closed for every point $x \in \O$. Now, we remark that the orbit $\PG \cdot \O$ is closed in $X$ if and only if the union $\PG \cdot (\O,\O) :=\underset{x \in \O}{\bigcup} \PG \cdot (\O,x)$ is closed in $X^{\bullet}$. Now, since there is a group $\G$ which divides $\O$, one can find a compact subset $K$ of $\O$ such that $\G \cdot K = \O$. Hence the conclusion follows from this computation:

$$\PG \cdot (\O,\O) =\underset{x \in K}{\bigcup} \underset{\g \in \G}{\bigcup}  \PG \cdot (\O,\g(K)) = \underset{x \in K}{\bigcup} \PG \cdot (\O,K).$$
\end{proof}

\section{The isometry group of a properly convex open set}

\subsection{The questions}
We basically know almost nothing about the group of isometries $\Isom(\O)$ of the metric space $(\O,d_{\O})$.

\begin{oqu}[Raised by de la Harpe in \cite{MR1238518}]
Is $\Isom(\O)$ a Lie group ? Is $\Aut(\O)$ always a finite-index subgroup of $\Isom(\O)$ ? If yes, does this index admit a bound $N_d$ depending only of the dimension ? Does $N_d=2$ ? 
\end{oqu}

The group $\Isom(\O)$ is a locally compact group for the uniform convergence on compact subsets by the Arzel\'a-Ascoli theorem.
We shall see that the answer to the first question is yes and it is already in the literature. The other questions are open and a positive answer would mean that the study of $\Aut(\O)$ or $\Isom(\O)$ is the same from a group-geometrical point of view.

The first question is a corollary of a wide open conjecture:

\begin{conjecture}[Hilbert-Smith]
A locally compact group acting effectively on a connected manifold is a Lie group.
\end{conjecture}

This conjecture is known to be true in dimension $d=1,2$ and $3$. For the proof in dimension $1$ and $2$, see \cite{MR0073104}, or Theorem 4.7 of \cite{MR1876932} for a proof in dimension 1. The proof in dimension $3$ is very recent and due to Pardon in \cite{Pardon:2011fk} in 2011.

The first article around this conjecture is due to Bochner and Montgomery \cite{MR0018187}, who proved it for smooth actions. Using a theorem of Yang (\cite{MR0120310}), Repov\u{s} and \u{S}\u{c}epin showed in \cite{MR1464908} that the Hilbert-Smith conjecture is true when the action is by bi-Lipschitz homeomorphism with respect to a Riemannian metric. There is also a proof of Maleshich in \cite{MR1480156} for H\"older actions. For a quick survey on this conjecture, we recommend the article of Pardon \cite{Pardon:2011fk}. The author of the present chapter is not an expert in this area and this paragraph does not claim to be an introduction to the Hilbert-Smith conjecture.

The proof of Repov\u{s} and \u{S}\u{c}epin works for Finsler metrics since it uses only the fact that the Hausdorff dimension of a manifold with respect to a Riemannian metric is the dimension of the manifold. We sketch the argument for the reader. We want to show that:

\begin{theorem}[Repov\u{s} and \u{S}\u{c}epin, \cite{MR1464908}]
A locally compact group acting effectively by Lipshitz homeomorphism on a Finsler manifold is a Lie group. In particular, the group $\Isom(\O,d_{\O})$ is a Lie group.
\end{theorem}

\begin{proof}
A classical and very useful reduction (see for example \cite{MR1641899}) of the Hilbert-Smith conjecture shows that we only have to show that the group $\Z_p$ of $p$-adic integers cannot act on $M$. Since every neighbourhood of the identity of $\Z_p$ contains a copy of $\Z_p$ we can assume that $\Z_p$ acts by Lipshitz homeomorphisms with bounded Lipschitz constant. Now, since $\Z_p$ is compact, we can assume that this action is in fact by isometries by averaging the metric using the Haar measure of $\Z_p$.

Now came the key argument of Repov\u{s} and \u{S}\u{c}epin. We have the following inequality where $\dim$, $\dim_H$ and $\dim_{\Z}$ mean respectively the topological dimension, the Hausdorff dimension and the homological dimension.

$$
\begin{array}{ccccccccc}
\dim(M) & = & \dim_H(M) &  \geqslant & \dim_H  \Big(\Quotient{M}{\Z_p}\Big)\\
\\
              & \geqslant & \dim \Big(\Quotient{M}{\Z_p}\Big) & \geqslant & \dim_{\Z}\Big(\Quotient{M}{\Z_p}\Big) & = &\dim(M)+2
\end{array}
$$

The first equality comes from the fact that the Hausdorff dimension of a Finsler manifold is equal to its topological dimension\footnote{The topological dimension of a compact topological space $X$ is the smallest integer $n$ such that every finite open cover $\mathcal{A}$ of $X$ admits a finite open cover $\mathcal{B}$ of $X$ which refines $\mathcal{A}$ in which no point of $X$ is included in more than $n+1$ elements of $\mathcal{B}$. If no such minimal integer $n$ exists, the space is said to be of infinite topological dimension. The topological dimension of a non-compact locally compact Hausdorff metric space is the supremum of the topological dimensions of its compact subspaces.}. The first inequality comes from the fact that since the action is by isometries the quotient map is a distance non-increasing map. The second and third inequalities come from the fact that we always have $\dim_H (X) \geqslant \dim(X) \geqslant \dim_{\Z}(X)$\footnote{For the definition of the homological dimension we refer to \cite{MR0120310}. For the inequalities, we refer to \cite{MR2321254} for the first and \cite{MR0120310} for the second.} when $X$ is a locally compact Hausdorff metric space. The last equality is the main result of Yang in \cite{MR0120310}.
\end{proof}

\subsection{The knowledge}

Since in this context subgroups of finite-index are important, we think that the group $\Aut^{\pm}(\O)$ of automorphisms of $\O$ of determinant $\pm 1$ is more adapted to the situation.

\begin{proposition}[de la Harpe in \cite{MR1238518}]
Let $\O$ be a properly convex open set. Suppose that the metric space $(\O,d_{\O})$ is uniquely geodesic. Then $\Aut^{\pm}(\O) = \Isom(\O)$.
\end{proposition}

\begin{proof}
If the only geodesics are the segments this implies that the image of a segment by an isometry $g$ is a segment. Therefore the fundamental theorem of projective geometry implies that $g$ is a projective transformation.
\end{proof}

The following proposition is a nice criterion for uniqueness of geodesics $(\O,d_{\O})$.

\begin{proposition}[de la Harpe \cite{MR1238518}]\label{seg_geo}
Let $\O$ be a properly convex open set. Then the metric space $(\O,d_{\O})$ is uniquely geodesic if and only if for every plane $\Pi$ intersecting $\O$, the boundary of the 2-dimensional convex set $\O \cap \Pi$ contains at most one maximal segment.
\end{proposition}

We stress the following corollary:

\begin{corollary}
If $\O$ is strictly convex, then $\Isom(\O)=\Aut^{\pm}(\O)$.
\end{corollary}

\begin{theorem}[de la Harpe \cite{MR1238518} for dimension 2, general case Lemmens-Walsh \cite{MR2819195}]
Suppose that $\O$ is a simplex. Then $\Aut^{\pm}(\O)$ is of index two in $\Isom(\O)$.
\end{theorem}

\begin{proof}
We give a rough proof in the case of a triangle $\O$. A nice picture will show you that the balls of the triangle are hexagons. Moreover, the group $\R^2$ acts simply transitively on the triangle $\O$. The induced map is an isometry from the normed vector space $\R^2$ with the norm given by the regular hexagon to $\O$ with the Hilbert distance.

Now, every isometry of a normed vector space is affine. So the group $\Isom(\O)$ has $12$ connected components since the group $\Quotient{\Isom(\O)}{\Isom^0(\O)}$ is isomorphic to the dihedral group of the regular hexagon, where $\Isom^0(\O)$ is the identity component of $\Isom(\O)$. But the group $\Aut^{\pm}(\O)$ has $6$ connected components and the group $\Quotient{\Aut^{\pm}(\O)}{Aut^{\pm 0}(\O)}$ is isomorphic to the dihedral group of the triangle.
\end{proof}

\begin{theorem}[Molnar complex case \cite{MR2529894}, general case Bosch\'e \cite{Bosche:2012uq}]
Suppose that $\O$ is a non-strictly convex symmetric properly convex open set. Then $\Aut^{\pm}(\O)$ is of index two in $\Isom(\O)$.
\end{theorem}

In the case where $\O$ is the simplex given by the equation $$\O=\{ [x_1:x_2:\dots : x_{d+1}] \,|\, x_i > 0\}$$
\noindent then an example of non-linear automorphisms is given by $$[x_1:x_2:\dots : x_{d+1}]  \mapsto [x_1^{-1}:x_2^{-1}:\dots : x_{d+1}^{-1}].$$ \noindent In the case where $\O$ is a non-strictly convex symmetric space, then $\O$ can be described as the projectivisation of a convex cone of symmetric definite positive matrices (or of positive Hermitian symmetric complex matrices, or the analogue with the quaternion or the octonion), and the non-linear automorphism is given by $M \mapsto ^t M^{-1}$ (or analogous).

In both cases, the non-linear automorphisms are polynomial automorphisms of the real projective space, the group $\Aut^{\pm}(\O)$ is normal in $\Isom(\O)$, even better, $\Isom(\O)$ is a semi-direct product of $\Aut^{\pm}(\O)$ with a cyclic group of order $2$.

Finally, Lemmens and Walsh show the following theorem:

\begin{theorem}[Lemmens and Walsh, \cite{MR2819195}]
Suppose that $\O$ is a polyhedron which is not a simplex. Then $\Isom(\O)=\Aut^{\pm}(\O)$.
\end{theorem}

Lemmens and Walsh conjecture that $\Isom(\O) \neq \Aut^{\pm}(\O)$ if and only if $\O$ is a simplex or is symmetric but not an ellipsoid. For a srvey on the isometry group of Hilbert geometry, see \cite{CW}.

\bibliographystyle{abbrv}

\end{document}